\documentclass[10.5pt]{article}
\usepackage{}
\usepackage{amsmath}
\usepackage{amsfonts}
\usepackage{graphicx}
\usepackage{amssymb}
\usepackage{color}

\newtheorem{condition**}{A*}
\newtheorem{condition***}{C*}
\newtheorem{condition*}{C}

\newtheorem{proposition}{Proposition}[section]

\newtheorem{definition}{Definition}[section]
\newtheorem{theorem}{Theorem}[section]
\newtheorem{lemma}{Lemma}[section]
\newtheorem{remark}{Remark}[section]

\def\qed{{\hfill $\Box$ \bigskip}}
\newenvironment{keywords}{{\bf Key words: }}{}

\def\d{\delta}
\def\bea{\begin{equation*}\begin{aligned}}
\def\eae{\end{aligned}\end{equation*}}

\begin{document}

\title{A unified approach to  mean-field team: homogeneity, heterogeneity and quasi-exchangeability}

\author{Xinwei Feng$\;^{a}$, Ying Hu$\;^{b}$, Jianhui Huang$\;^{c}$\bigskip\\{\small ~$^{a}$Zhongtai Securities Institute for Financial Studies, Shandong University, Jinan, Shandong
250100, China}\\{\small ~$^{b}$Univ Rennes, CNRS, IRMAR-UMR 6625, F-35000 Rennes, France}\\{\small $^{c}$Department of Applied Mathematics, The Hong Kong Polytechnic University, Hong Kong, China}}
\renewcommand{\thefootnote}{\fnsymbol{footnote}}

\renewcommand{\thefootnote}{\arabic{footnote}}
%\footnotetext[1]{{\scriptsize
%The work of Xinwei Feng is partially supported  by National Natural Science Foundation of China (No. 12001317), Shandong Provincial Natural Science Foundation (No. ZR2020QA019) and QILU Young Scholars Program of Shandong University.
%       The work of Ying Hu is partially supported  by  Lebesgue Center of Mathematics ``Investissementsd'avenir''program-ANR-11-LABX-0020-01, by ANR CAESARS (Grant No. 15-CE05-0024) and by ANR MFG (Grant No. 16-CE40-0015-01). The first and third author also acknowledge financial support from: PolyU-SDU Joint Research Centre on Financial Mathematics, PolyU JRI Postdoc scheme and G-UAAQ.
%       }}
%\renewcommand{\thefootnote}{\fnsymbol{footnote}}
%\footnotetext{\textit{{\scriptsize E-mail:}}
%  {\scriptsize xwfeng@sdu.edu.cn (Xinwei\ Feng); ying.hu@univ-rennes1.fr (Ying\ Hu); majhuang@polyu.edu.hk (Jianhui\ Huang).}}

%\date{January 31, 2021}
\maketitle

\begin{abstract}This paper aims to systematically solve stochastic team optimization of large-scale system, in a rather general framework. Concretely, the underlying large-scale system involves considerable weakly-coupled cooperative agents for which the individual admissible controls: (\textbf{i}) enter the diffusion terms, (\textbf{ii}) are constrained in some closed-convex subsets, and (\textbf{iii}) subject to a general \emph{partial decentralized information} structure. A more important but serious feature: (\textbf{iv}) all agents are heterogenous with \emph{continuum} instead \emph{finite} diversity. Combination of (\textbf{i})-(\textbf{iv}) yields a quite general modeling of stochastic team-optimization, but on the other hand, also fails current existing techniques of team analysis. In particular, classical team consistency with continuum heterogeneity collapses because of (\textbf{i}). As the resolution, a novel \emph{unified approach} is proposed under which the intractable \emph{continuum} \emph{heterogeneity} can be converted to a more tractable \emph{homogeneity}. As a trade-off, the underlying randomness is augmented, and all agents become (quasi) weakly-exchangeable. Such approach essentially involves a subtle balance between homogeneity v.s. heterogeneity, and left (prior-sampling)- v.s. right (posterior-sampling) information filtration. Subsequently, the consistency condition (CC) system takes a new type of forward-backward stochastic system with \emph{double-projections} (due to (\textbf{ii}), (\textbf{iii})), along with \emph{spatial mean} on continuum heterogenous index (due to (\textbf{iv})). Such system is new in team literature and its well-posedness is also challenging. We address this issue under mild conditions. Related asymptotic optimality is also established.
\end{abstract}


\section{Introduction}

The starting point of present work is the well-studied mean-field team (MT). In its standard form, a MT involves a large-scale system with considerable weakly-interactive but \emph{cooperative} agents $\{\mathcal{A}_{i}\}_{i=1}^{N}.$ All agents are endowed with an individual (\emph{principal}) state, cost functional and admissible decision set respectively in the following manner. The individual state dynamics of $\mathcal{A}_{i}$ is formulated by a controlled It\^{o}-type linear stochastic differential equation (LSDE):
\begin{equation}\label{homo-state}\left\{\begin{aligned}
dx_i(t)=&[A(t)x_i(t)+B(t)u_i(t)+F(t)x^{(N)}(t)+f_t]dt+\sigma_t dW_i(t),\\
x_i(0)=&\xi\in\mathbb R^n,\qquad 1\leq i\leq N,
\end{aligned}\right.\end{equation}where $x^{(N)}:=\frac{1}{N}\sum_{i=1}^Nx_i$ is the weakly-coupled state-average across all agents, $W_i$ is a Brownian Motion (BM) that might be vector-valued (e.g., with a common noise). For each $\mathcal{A}_{i},$ its \emph{principal cost} $\mathcal{J}_{i}$ (while we may call $\{\mathcal{J}_{j}\}_{j \neq i}$ the \emph{marginal costs} for $\mathcal{A}_{i}$) is measured by the following quadratic functional:
 \begin{equation}\label{cost}\begin{aligned}
\mathcal J_i(\mathbf{u}(\cdot))
=\frac{1}{2}\mathbb E\int_0^T\Big[\langle Q(t)(x_i(t)-H(t) x^{(N)}(t)),x_i(t)-H (t) x^{(N)}(t)\rangle+\langle R(t)u_i(t),u_i(t)\rangle\Big]dt,
\end{aligned}\end{equation}with admissible team strategy $\mathbf{u}(\cdot)=(u_1^\top(\cdot),\cdots,u_N^\top(\cdot))^\top.$ Note individual admissible $u_i(\cdot) \in \mathcal{U}^{d,f}_{i,op}= L^{2}_{\mathbb{F}^i}(0,T;\mathbb R^m)$ with filtration $\mathbb{F}^i$ defined later, representing the decentralized open-loop information of $\mathcal A_i$.

A subtle point here is the distinction between centralized ($\mathcal U_i^{c,f}$), and decentralized ($\mathcal U_{i,op}^{d,f}$, $\mathcal U_{i,cl}^{d,f}$) but of full information. This makes team-optimization differing from classical \emph{vector}-optimization/control; superscripts ``\emph{cl}", ``\emph{ol}" denote the closed-loop and open-loop; ``\emph{f}" the full-information. We will address this point more detailed in Section \ref{formulation}. Hereafter, we may exchange the usage of $\mathbf u=(u_1,\cdots,u_N)\in\mathbb R^{m\times N}$, $\mathbf{u}=(u_1^\top,\cdots,u_N^\top)^\top\in\mathbb R^{mN}$ and $\mathbf u=(u_i,u_{-i})\in\mathbb R^{m\times N}$ with $u_{-i}=(u_1,\cdots,u_{i-1},u_{i+1},u_N)\in\mathbb R^{m\times(N-1)}$ by noting all of them represent team profile among all agents, but only differ in formations. For simplicity, we focus on \emph{Lagrange problem} only, and no essential difficulty to \emph{Bolza problem} extension.

By mean-field ``team", we refer all weakly-coupled agents $\{\mathcal{A}_{i}\}_{i=1}^{N}$ are cooperative aiming to optimize the following social (or, team) cost functional (the related optimal functional is called \emph{social optima}):\begin{equation*}
\mathcal J_{soc}^{(N)}(\mathbf{u}(\cdot))=\sum_{i=1}^N\mathcal J_i(\mathbf{u}(\cdot)).
\end{equation*}
Because of the cooperation nature, the analysis of MT should proceed very differently from that of mean-field game (e.g., \cite{BP,BSYS2016,Cardaliaguet,CD2013,CS2015,LL2007}), especially in its analysis ingredients on variational decomposition and \emph{person-by-person optimality} principle. For non-cooperate $N$-player game with interaction of mean field type, the objective of the players is to seek the Nash equilibria. Please refer \cite{CF,CDLL,Fisher,Lacker,NST} for the limit relation between mean-field games (MFG) and non-cooperate $N$-player games. The interested readers may refer e.g., \cite{HCM2012,NCMH,SLM}, for detailed analysis comparison between MFG and MT, and \cite{QHX,WZZ} for some recent MT study from various perspectives with different modeling variants. In particular, see \cite{HWY2019} for social optima in mean field control problems with volatility uncertainty; see \cite{Huang2010} for linear-quadratic-Gaussian (LQG) mean-field social optimization with a major player; and \cite{WZ2017} for social optima in LQG models with Markov jump parameters.

Our work distinguishes itself from all above MT literature by the following fairly (even not the most) general formulation, in LQG context. Unlike \eqref{homo-state}, the individual dynamic of agent $\mathcal{A}_{i}$ now takes:
 \begin{equation}\left\{\label{state equation}\begin{aligned}
dx_i(t)=&[A_{\Theta_i}(t)x_i(t)+B(t)u_i(t)+F(t)x^{(N)}(t)]dt\\
&+[C(t)x_i(t)+D_{\Theta_i}(t)u_i(t)+\widetilde  F(t)x^{(N)}(t)]dW_i(t),\\
x_i(0)=&\xi\in\mathbb R^n,\qquad 1\leq i\leq N,
\end{aligned}\right.\end{equation}
and the admissible strategy set for $\mathcal A_i$ is now assumed to be
\begin{equation}\label{admissible control}
\mathcal U_i^{d,p}=
\{u_i(\cdot)|u_i(\cdot)\in L^2_{\mathbb G^i}(0,T;\Gamma)
\}
\end{equation}
 where $\mathbb G^i \subseteq\mathbb F^i$ or $\mathbb G^i \subseteq \mathbb H^i$
  is a sub-filtration representing the partial information; $\Gamma \subset \mathbb R^m$ is a nonempty closed convex set representing the input constraint.

There are four main modeling features in formulation \eqref{state equation}, \eqref{admissible control}:

(\textbf{i})\emph{Weakly-coupled controlled-diffusion}. It is remarkable that in \eqref{homo-state}, when $D \neq 0$ so control process enters diffusion terms of It\^{o}-type LSDE (driven by $W_i(\cdot)$), and when $\widetilde  F \neq 0$ so all individual states are weakly-coupled in diffusion terms also. In this case, we may call (\ref{state equation}) to be \emph{diffusion-controlled and weakly-coupled}. This differs from \cite{HCM2012} in modeling that is only \emph{drift-controlled and weakly-coupled}. Such modeling difference also brings considerable analysis distinctions, for example, on the relevant study of  Hamiltonian systems, as well as consistency condition (CC) (see more comparison details in Section \ref{auxiliary problem} and Section \ref{decentralized strategy}). Without loss of generality, no forcing terms such as $f, \sigma$ involve in \eqref{state equation}.

\textbf{(ii)} {\emph{Random diversity}}. Recall that \eqref{homo-state} is \emph{homogenous} since all agents are endowed with identical parameters thus they become symmetric. Subsequently, the (decentralized) optimal strategy and states, still denoted as $\{u_i\}_{i=1}^{N}$ and $\{x_i\}_{i=1}^{N}$, should turn to be exchangeable. By contrast, in \eqref{state equation}, a random index $\Theta_i $ is introduced in parameter $A, D$  (also possible to be equipped on other parameters including the cost) to model the diversity across underlying large-scale system. All agents thereby become heterogenous. Although heterogenous large-scale system is already well addressed in such as \cite{HHN2018,Huang2010}, we point out in these works, the heterogenous index is technically treated as some realization after random sampling, along with necessary \emph{ordinal arrangements} within each sub-classes. Thus, essentially the index therein is some deterministic realization. This differs substantially from our random index treatment here along with related analysis, to be highlighted later. In addition, our index $\Theta_i$ can assume a \emph{continuum} support that distinguishes from most heterogenous literature with only finite/discrete support (see., e.g., \cite{HHN2018,Huang2010}). Moreover, although continuum heterogeneity is also discussed in e.g., \cite{NH}, but analysis therein heavily relies upon the LQ structure with full input and resultant explicit representation. Such analysis collapses in current formulation \eqref{state equation}, due to the intrinsic \emph{diffusion-controlled weakly-coupled} feature introduced before, and an input constraint feature to be introduced below.

(\textbf{iii}) {\emph{Input constraint}}. Note that a convex-closed set $\Gamma$ is introduced in \eqref{admissible control} denoting some point-wise constraint in control input. Recall that such pointwise input constraint is well documented in e.g., \cite{CZ2006,ET2015,HZ2005,LZL2002}. A typical example is $\Gamma=\mathbb{R}^{+}$ representing the positive control, or \emph{no-shorting} constraint in portfolio selection (\cite{LZL2002}). Other examples may include subspace (\cite{ET2015}) or a general convex cone (\cite{HZ2005}). We remark that point-wise input constraint is also studied in large-scale/large-population context such as \cite{HHL2017} but in competitive \emph{mean-field-game} setup, which differs from our cooperative mean-field team here.

(\textbf{iv}) {\emph{Partial information}}. Last but not least, the admissible control set is confined on a partial information set $L^{2}_{\mathbb{G}^i}(0,T;\Gamma)$. LQG control with partial information is also well documented (e.g., \cite{WWX2015}). Also, partial information for large population system is also addressed recently (see \cite{BLM,CK2017,FC2020,HWW2016} for partial information/observation mean-field game). However, to our best knowledge, it is the first time to address partial information in \emph{mean-field team} context. Notice that the partial information setting differs from that of partial observation (\cite{Bensoussan}) for which some filtering method with innovation process should be invoked. We defer more detailed information structure in Section \ref{formulation} after more rigorous formulation.

To certain content, our aim in current work is to solve LQG MT problem in a rather general setup, by combining aforementioned features (i)-(iv) together. Although we admit various effective techniques have been already proposed to tackle these features \emph{individually}, however their \emph{combination} brings much more technical hurdles, and makes the associated analysis rather challenging. For example, the continuum heterogenous large-scale system is well studied by \cite{NH} in mean-field game setup. Nevertheless, its parallel analysis variant to MT fails to work in current formulation because of the following reasoning. Due to controlled-diffusion feature (i), the related CC does not admit direct characterization because the adjoint process of some backward SDE should enter CC dynamics. Therefore, the direct augmented method in \cite{WZZ} fails to work here. Instead, some indirect embedding method \cite{HHN2018,QHX} becomes necessary in the presence of (i). Nevertheless, due to continuum heterogenous feature (ii), the classical embedding CC in \cite{HHN2018,QHX} no longer works since we have to construct an infinite-dimensional Brownian motion-driven system (on continuum-valued space) to replicate the empirical distribution generated by controlled large-scale system. Meanwhile, the method in \cite{QHX} is also not infeasible since it mainly rely on some close-form representation of optimal state/cost. This becomes unavailable because of the input constraint (iii) imposed above. In nutshell, in case (i) or (iii) not combined togeher, we may still handle continuum heterogenous MT with (ii) by modifying existing methods in e.g., \cite{QHX}. However, combination of (i), (ii), (iii) together make all such existing methods no longer workable.

Other examples include the person-by-person procedure due to continuum heterogeneous (ii), and tailor-made decentralized strategy in presence of both point-wise constraint (iii) and partial information constraint (iv). To circumvent these difficulties, we propose some novel analysis techniques such as weak construction duality and modified embedding representation, etc. More analysis details are illuminated in Section \ref{auxiliary problem} and Section \ref{well-posedness of CC}.

Our main contributions can be sketched as follows: (1) First, we devise a new framework to unify homogenous and heterogenous (discrete or continuum) setups in large-scale system. In particular, it is enabled to transform heterogenous setup into a homogenous one, with the tradeoff of an augmented randomness. (2) Second, under such new framework, we derive a modified embedding representation of CC system (a crux in MT analysis) to accommodate the continuum diversities. (3) Third, the input constraint and partial information constraint are tackled both, and a CC system with \emph{double projection} operator is derived. Specifically, the CC system takes a coupled mean-field type forward backward stochastic differential equations (FBSDEs) involving both projection mapping and conditional expectation. This seems quite novel in large-scale literature. (4) Last, the well-posedness of CC system and asymptotic team optimality are established under mild conditions.

We would like to conclude above discussion by highlighting a literature comparison. Seemingly, the current work seems closely related to previous work \cite{HHN2018}. However, the formulation of \cite{HHN2018} is non-cooperative mean-field game with finite heterogenous diversity. By contrast, the current work focuses on a cooperative mean-field team with continuum random diversity index. In addition, current formulation includes partial information also, thus the CC condition here involves a double projection whereas \cite{HHL2017,HHN2018} only involves one single projection. Last but not least, other MT analysis ingredients also differ essentially from those in MG setup such as \cite{HHN2018}, owning to the intrinsic distinction between game and team.

The remaining of this paper is organized as follows. In Section 2, we give the formulation of LQG heterogeneous agents problem with input constraints and partial information pattern. In Section 3, we apply person-by-person optimality and weak construction duality to find the auxiliary control problem of the individual agent. The decentralized strategy and consistency condition is established in Section 4. Moreover, we also compare our framework with that in the current literature. Section 5 studies the well-posedness of CC system, asymptotic optimality of decentralized strategy is given in Section 6.

\section{Problem formulation}\label{formulation}

We first introduce some standard notations used throughout this paper. Let $\mathbb R^n$ be the $n$-dimensional Euclidean space with the inner product denoted by $\langle\cdot,\cdot,\rangle$. $\mathbb R^{n\times m}$ is the space of all $(n\times m)$ matrices, endowed with the inner product $\langle M_1,M_2\rangle=tr[M_1^\top M_2]$, where $x^\top$ denotes the transpose of a matrix (or vector) $x$ and $tr$ is the trace of a matrix.
$M\in \mathbb{S}^n$ denotes the set of symmetric $n\times n$ matrices with real elements. $M> (\geq) 0$ denotes that $M\in \mathbb{S}^n$ which is positive (semi)definite, while $M\gg 0$ denotes that, $M-\varepsilon I \geq 0$ for some $\varepsilon>0.$

Assume that $(\Omega,\mathcal F,\mathbb P)$ is a complete probability space on which $\{W_i(t),0\leq t\leq T\}_{i=1}^N$ is a $N$-fold Brownian motion (note here $W_i$ might be vector-valued, say, including a common noise component $W_0$) and $\{\Theta_i\}_{i=1}^N$ is a sequence of independent random variables
to represent diversity. In some sense, we may interpret $\{\Theta_i\}$ as some
endogenous randomness, while $\{W_i\}$ some exogenous randomness for generic agent $\mathcal{A}_{i}.$ Moreover, we assume $\{\Theta_i\}_{i=1}^N$ are also independent of $\{W_i(s),s\geq0\}_{i=1}^N$. Let $\{\mathcal F_t^W\}_{0 \leq t \leq T}$ be the filtration generated by $\{W_i(s),0\leq s\leq t\}_{i=1}^N$ and define $\mathcal F^{W,\Theta}_t=\sigma(\Theta_i,1\leq i\leq N)\vee\mathcal F^W_t$. The set of null sets on $\Omega$ is defined by $\mathcal N_{\mathbb P}=\{M\in\Omega|\exists G\in\mathcal F^{W,\Theta}_\infty\text{ with }M\subset G\text{ and }\mathbb P(G)=0\}$. Consider the augmented filtration $\mathbb F=\{\mathcal F_t\}_{0\leq t\leq T}$ with $\mathcal F_t=\sigma(\mathcal F^{W,\Theta}_t\cup\mathcal N_{\mathbb P})$.
Similarly, define $\mathcal F_t^{W_i},\mathcal F_t^{W_i,\Theta_i},\mathcal F_t^i$. As discussed below, they respectively denote the centralized and decentralized information.

For any Euclidean space $\mathbb V$, $1\leq p<\infty$, and any $T>0$, we introduce some spaces which will be used later:
\begin{itemize}
  \item $L_{\mathcal F_T}^p(\Omega;\mathbb V):=\{\eta:\Omega\rightarrow\mathbb V|\eta\text{ is }\mathcal F_T\text{-measurable such that }\mathbb E|\eta|^p<\infty\}$.
  \item $L^\infty(0,T;\mathbb V):=\{\varphi(\cdot):[0,T]\rightarrow\mathbb V\text{ such that } esssup_{0\leq s\leq T}|\varphi(s)|<\infty\}.$
  \item $L^p(0,T;\mathbb V):=\{\varphi(\cdot):[0,T]\rightarrow\mathbb V\text{ such that } \int_0^T|\varphi(s)|^pds<\infty\}.$
   \item $L^p_{\mathbb F}(0,T;\mathbb V):=\{\varphi(\cdot):\Omega\times[0,T]\rightarrow\mathbb V\text{ is progressively measurable such that}\\ \mathbb E\int_0^T|\varphi(s)|^pds<\infty\}.$
\end{itemize}
We consider a weakly coupled large population system of heterogeneous agents $\{\mathcal A_i:1\leq i\leq N\}$ with the dynamics of the agents given in \eqref{state equation}, and cost functional \eqref{cost}. For sake of presentation, we restate them as follows:
\begin{equation}\label{LQGMT-1}\left\{\begin{aligned}
&\left\{\begin{aligned}dx_i(t)=&[A_{\Theta_i}x_i+Bu_i+Fx^{(N)}]dt+[Cx_i+D_{\Theta_i}u_i+\widetilde  Fx^{(N)}]dW_i,\\
x_i(0)=&\xi\in\mathbb R^n,\qquad 1\leq i\leq N,\\
\end{aligned}\right.\\
&\mathcal J_i(\mathbf{u}(\cdot))
=\frac{1}{2}\mathbb E\int_0^T\Big[\langle Q(x_i-H x^{(N)}),x_i-H  x^{(N)}\rangle+\langle Ru_i,u_i\rangle\Big]dt.
\end{aligned}\right.\end{equation}
As mentioned before, state \eqref{state equation} and functional \eqref{cost} formulate a weakly coupled large-scale system with heterogeneous agents $\{\mathcal A_i:1\leq i\leq N\}.$ The aggregate team functional of $N$ agents is
\begin{equation*}
\mathcal J_{soc}^{(N)}(\mathbf{u}(\cdot))=\sum_{i=1}^N\mathcal J_i(\mathbf{u}(\cdot)).
\end{equation*}
$(A_{\Theta_i}(\cdot),B(\cdot),C(\cdot),D_{\Theta_i}(\cdot),F(\cdot),\widetilde F(\cdot))$ are called the state-coefficient datum, while $(Q(\cdot),\\H(\cdot),R(\cdot))$ the cost weight datum.
We explain more details for above datum. $F, \widetilde{F}$ are \emph{weakly-coupling coefficients} on state-drift and state-diffusion respectively; $H$ is weakly-coupling coefficient on functional; $C, D_{\Theta_i}$ are diffusion state-dependence and  diffusion control-dependence coefficients respectively. Note that $D_{\Theta_i} \neq 0$ represents the case when control enters diffusion terms alike the risky portfolio selection (e.g., \cite{HZ2005,LZL2002,zl2000}); $F, \widetilde{F} \neq 0$ denotes the agents are coupled in their dynamics such as the price formation problem (e.g., \cite{GD,LLLL}); $H \neq 0$ denotes the \emph{relative performance} formulation (e.g., \cite{ET2015}).

Unlike state \eqref{homo-state}, we introduce $\{\Theta_i\}_{i=1}^{N}$ in \eqref{state equation} as some diversity index to characterize the possible heterogenous features among all agents in underlying large-scale system. We point out that $\Theta_i$ maybe vector-valued on a Cartesian \emph{grid} space, say $[a_1, b_1] \times [a_2, b_2]$ or $[a_1, b_1] \times \{1, \cdots, K\}$, to represent various feature dimensions, either in continuum space or discrete space, or in a hybrid manner.

\begin{remark}We remark that discrete- or finite-valued $\Theta_i$ might be transformed into continuum one by assigning uniform distribution on compact interval along with given partitions. Indeed, this is equivalent to simulate a given discrete random variable using quantile method by uniform distribution. Thus, hereafter we focus on vector-valued index $\Theta_i$ on Cartesian space $\mathbb{R}^{k}$.\end{remark}For simplicity, we only assume that the coefficients $A$ and $D$ to be dependent on $\Theta_i$. Similar analysis can be generalized to the case when all other coefficients are also $\Theta_i$-dependent. Besides, all datum may depend on time variable $t$, in what follows the variable $t$ will usually be suppressed if no confusion occurs.  We now introduce the following assumption on distribution and coefficient datum set:
\begin{description}
  \item[(A1)] For $i=1,\cdots,N$, $\Theta_i:\Omega\rightarrow \mathcal{S}$ are independently identically distributed (i.i.d) with the distribution function $\Phi(\theta)$, i.e., $\int_{\mathcal{S}} d\Phi(\theta)=1$, where $\mathcal{S}$ is a continuum subset in Cartesian space $\mathbb{R}^{k} $.

  \item[(A2)] For any $\theta\in\mathcal S$, $A_{\theta}(\cdot),F(\cdot),C(\cdot),\widetilde F\in L^\infty(0,T;\mathbb R^{n\times n}),B(\cdot),D_{\theta}(\cdot)\in L^\infty(0,T;\mathbb R^{n\times m}),\\Q(\cdot)\in L^\infty(0,T;\mathbb S^n)$, $H(\cdot)\in L^\infty(0,T;\mathbb S^n)$, $R(\cdot)\in L^\infty(0,T;\mathbb S^m)$.
\item[(A3)]$Q(\cdot)\geq0$, $R(\cdot)\gg0$.
\end{description}
Under assumptions (A1)-(A2), the state \eqref{state equation} admits a unique strong solution $$x(\cdot)=(x_1(\cdot),\cdots,x_N(\cdot)) \in L^2_{\mathbb F}(0,T;\mathbb R^{N \times n}) ,$$ and the cost functional is well defined for each admissible control strategy $\mathbf{u} (\cdot)$ on appropriate admissible space, to be detailed soon. Moreover, under assumption (A3), the cost functional is uniform convex, that is, there exists some $\delta>0$ such that  $\mathcal J_{soc}^{(N)}(\mathbf{u}) \geq \delta \mathbb E\int_0^T|\mathbf{u}(s)|^{2}ds$.

Given state \eqref{state equation} and functional \eqref{cost}, we can specify the associated information structures.
Recall that in LQG MT, $\{x_{i}\}_{i=1}^{N}$ and $\{u_{i}\}_{i=1}^{N}$ denote states and controls of $\{\mathcal{A}_i\}_{i=1}^{N}$ respectively. Because of interactive coupling by state-average $x^{(N)}:=\frac{1}{N}\sum_{i=1}^{N}x_{i}$, $\mathcal{J}_{i}(u_i, u_{-i})$ depends on total team-decision $\mathbf{u}=(u_i, u_{-i})$. In this sense, \eqref{state equation} exhibits the so-called \emph{weakly interactive coupling} in decision when $N\rightarrow +\infty.$ Again, by such interactive coupling, information structure of \eqref{state equation} becomes more involved:
\begin{itemize}

\item \emph{Centralized information}: consider the filtration $\mathcal F_t^W=\sigma(W_i(s),0\leq s\leq t,i=1,\cdots,N)$, $\mathcal F^{W,\Theta}_t=\sigma(\Theta_i,1\leq i\leq N)\bigvee\mathcal F^W_t$, $0\leq t<\infty$, as well as the set of null sets $\mathcal N_{\mathbb P}=\{M\in\Omega|\exists G\in\mathcal F^{W,\Theta}_\infty\text{ with }M\subset G\text{ and }\mathbb P(G)=0\}$, and create the augmented  filtration $\mathbb F=\{\mathcal F_t\}_{0\leq t\leq T}$ with $\mathcal F_t=\sigma(\mathcal F^{W,\Theta}_t\cup\mathcal N_{\mathbb P})$.
Then $\mathbb{F}=\{\mathcal F_t \}_{0 \leq t \leq T} $ represents the centralized information including all Brownian motions (BMs) and diversity index components across all agents (principal and marginals).

\item \emph{Decentralized, open-loop information}: consider the filtration $\mathcal F_t^{W_i}=\sigma(W_i(s),0\leq s\leq t)$, $\mathcal F^{W_i,\Theta_i}_t=\sigma(\Theta_i)\bigvee\mathcal F^{W_i}_t$, $0\leq t<\infty$, as well as the set of null sets $\mathcal N_{\mathbb P}^i=\{M\in\Omega|\exists G\in\mathcal F^{W_i,\Theta_i}_\infty\text{ with }M\subset G\text{ and }\mathbb P(G)=0\}$, and create the augmented  filtration $\mathbb F^i=\{\mathcal F_t^i\}_{0\leq t\leq T}$ with $\mathcal F_t^i=\sigma(\mathcal F^{W_i,\Theta_i}_t\cup\mathcal N_{\mathbb P}^i)$. Then $\mathbb{F}^i$ represents the decentralized open-loop information that only includes the \emph{principal} components for $\mathcal A_i$. Note that $\{\mathcal{F}^{i}_t\}$ only depends on underlying $W^{i}$ and $\Theta_{i}$ instead of state $x_i$ itself, thus we call it open-loop (although it also differs from classical open-loop due to mean-field nature) information since it depends directly on underlying randomness.

    \item \emph{Decentralized, closed-loop information}: denote by $\{\mathcal{H}_t^{i}\}_{0 \leq t \leq T}$ the filtration by individual state $x_{i}$ augmented by $\mathcal N_{\mathbb P}^i$, i.e., $\mathcal{H}_t^{i}=\sigma\{x_i(s), 0 \leq s \leq t\}\bigvee\mathcal N_{\mathbb P}^i$, then $\mathbb{H}^{i}:=\{\mathcal{H}_t^{i}\}_{0 \leq t \leq T}$ represents decentralized closed-loop information; Note that $\{\mathcal{H}^{i}_t\}$ only depends on underlying principal state $x_i$ itself, thus we call it closed-loop (although it also differs from classical closed-loop due to mean-field nature). We remark that $x_{i}$ is not adapted to $W^{i}$ and $\Theta_{i}$ due to weakly coupling.

     \item  \emph{Decentralized, partial information}: Let $\mathcal G^i_t \subseteq \mathcal F^i_t$  be a sub-$\sigma$-field of $\mathcal F^i_t$ (or, $\mathcal G^{i}_t \subseteq \mathcal H^{i}_t $ be a sub-$\sigma$-field of $\mathcal H^i_t$), then $\mathbb{G}^i=\{\mathcal G^i_t\} _{0 \leq t \leq T} $ represents the decentralized partial information (open-loop or closed-loop) available to $\mathcal A_i$.\end{itemize}
\begin{remark}\label{re2.2}
For decentralized, partial information pattern, $\mathcal G^i_t$ is a given filtration representing
the information available to $\mathcal A_i$ at time t. For example, $\mathcal G^i_t=\mathcal F^i_{(t-\delta)+}$, or $\mathcal G^i_t=\mathcal H^i_{(t-\delta)+}$, $t\in [0, T] $,
where $\delta > 0$ denotes the fixed delay of information. In this case, $\mathcal G^i_t$ represent the partial information in open-loop or closed-loop sense, respectively. Another example is that $W_i=(\widetilde W_i,\widetilde W_0)$ takes vector-valued Brownian motion including a common noise component $\widetilde W_0$, then $\mathcal G^i_t=\sigma\{\widetilde W_{i}(s),\Theta_i, 0 \leq s \leq t\}$ denotes the partial information in open-loop. Also, in case $\Theta_i=(\Theta_{i1}, \Theta_{i2})$, then $\mathcal G^i_t=\sigma\{ W_{i}(s),\Theta_{i1}, 0 \leq s \leq t\}$ denotes the partial information to underlying diversity.

\end{remark}Therefore, $\mathcal{B}_t^{i}=\mathcal{F}_t^{i} \bigvee \mathcal{H}_t^{i}$ and $\mathbb{B}^{i}:=\{\mathcal{B}^{i}_t\}_{0 \leq t \leq T}$ represents (full) decentralized information. Then we have the following structure inclusion chart:
\begin{equation*}
 \mathbb G^i\subset\{ \mathbb{F}^{i}({\footnotesize{\text{decentralized open-loop}}}), \ \mathbb{H}^{i}({\footnotesize{\text{decentralized closed-loop}}})\}  \ \subset \ \mathbb{B}^{i}({\footnotesize{\text{decentralized}}}) \  \subset \  \mathbb{F} \ ({\footnotesize{\text{full}}}).\vspace{-0.2cm}\end{equation*}
Noticing due to state-average $x^{(N)}$, $x_{i}(t) \notin \mathcal{F}^{i}_t$, thus, NO inclusion relations between open-loop $\mathbb{F}^{i}=\{\mathcal{F}^{i}_t\}_{0 \leq t \leq T}$ and closed-loop $\mathbb{H}^{i}=\{\mathcal{H}^{i}_t\}_{0 \leq t \leq T}$.
This is different to classical control where the open-loop information includes closed-loop information. Given information structure, we are ready to formulate the relevant admissible control sets:
\begin{itemize}
  \item \text{Centralized full-information  admissibility set:} \ $\mathcal{U}_{i}^{c,f}=
\{u_i(\cdot)|u_i(\cdot)\in L^2_{\mathbb F}(0,T;\Gamma)\}$.
  \item \text{Decentralized full-information open-loop admissibility set:}\\\text{\qquad \qquad}\qquad$\mathcal{U}_{i,op}^{d,f}=
\{u_i(\cdot)|u_i(\cdot)\in L^2_{\mathbb F^i}(0,T;\Gamma)\}$.
\item \text{Decentralized full-information  closed-loop admissibility set:}\\\text{\qquad \qquad} \qquad $\mathcal{U}_{i,cl}^{d,f}=
\{u_i(\cdot)|u_i(\cdot)\in L^2_{\mathbb H^i}(0,T;\Gamma)\}$.
  \item \text{Decentralized partial-information admissibility set:} \ $\mathcal{U}_{i}^{d,p}=
\{u_i(\cdot)|u_i(\cdot)\in L^2_{\mathbb G^i}(0,T;\Gamma)\}$.
\end{itemize}
We point out here $\mathbb G^{i}$ is general to include both open-loop or closed-loop partial information. Now we propose the following optimization problem:\\

\textbf{Problem LQG-MT.} Find a team strategy set $\bar{\mathbf{u}}(\cdot)=(\bar u_1(\cdot),\cdots,\bar u_N(\cdot))$ where $\bar u_i(\cdot)\in \mathcal U_i^{c,f}$, $1\leq i\leq N$, such that
\begin{equation*}
\mathcal J_{soc}^{(N)}(\bar{\mathbf{u}}(\cdot))=\inf_{u_{i}\in\mathcal U_{i}^{c,f},1\leq i\leq N}\mathcal J_{soc}^{(N)}(u_1(\cdot),\cdots,u_i(\cdot),\cdots,u_{N}(\cdot)).
\end{equation*}
Under some mild conditions on datum $(Q, R)$ (e.g., (A3)), it is possible to ensure the existence and uniqueness of optimal mean-field team strategy in a centralized sense. This can be proceeded by classical vector-optimization or control method but in a high-dimension setting because of the existence of large number of weakly-coupled team agents. However, such strategy, from a computational viewpoint, turns to be intractable because of the information requirement to collect all agents' states simultaneously. Instead, it is more tractable to consider some decentralized strategy for which only the local (distributed) information for given agent is needed. Moreover, considering the partial information pattern, we introduce the following definition on asymptotic social optimality.

\begin{definition}
A strategy set $\widetilde{\mathbf{u}}(\cdot)=(\widetilde u_1(\cdot),\cdots,\widetilde u_N(\cdot))$ with $\{\widetilde u_i\in\mathcal U_i^{d,p}\}_{i=1}^N$ is said to be $\varepsilon$-social optimal if there exists $\varepsilon=\varepsilon(N)>0$,
$\displaystyle{\lim_{N\rightarrow+\infty}}\varepsilon(N)=0$ such that
$$\frac{1}{N}(\mathcal J_{soc}^{(N)}
(\widetilde{\mathbf{u}}(\cdot))-\inf_{u\in\mathcal U_i^{c,f}}\mathcal J_{soc}^{(N)}(\mathbf{u}(\cdot)))\leq\varepsilon.$$
\end{definition}
\begin{remark}
In Remark \ref{re2.2}, we emphasize $W_i$ might be vector-valued Brownian motion including a common noise component. For simplicity, in the following we assume that $W_i$, $i=1,\cdots,N$ are independent one-dimensional Brownian motions. Note that for the case $W_i=(\widetilde W_i,\widetilde W_0)$ takes vector-valued Brownian motion including a common noise component $\widetilde W_0$ and $\widetilde W_i$, $i=1,\cdots,N$ being independent one-dimensional Brownian motions, the procedures in Section \ref{auxiliary problem} and Section \ref{decentralized strategy} are still workable. However, in this case $\mathbb E\alpha$ in \eqref{CC} should be the conditional expectation $\mathbb E[\alpha|\mathcal F^0_t]$ where $\{\mathcal F^0_t\}$ is the filtration generated by the common noise $\widetilde W_0$. For this kind of consistency system, please refer \cite{HHN2018} for more information.
\end{remark}

\section{Mean-field team analysis}\label{auxiliary problem}As discussed above, the centralized strategy based on traditional vector optimization/control,  turns to be inefficient to tackle the weakly-coupled but highly complex LQG MT. Alternatively, it is more desirable to construct some decentralized strategy using distributed information only. Such strategy construction might be proceeded using mean-field team analysis through the following steps:

(Step 1) applying person-by-person optimality to variational decomposition for generic agent;

(Step 2) constructing some auxiliary control problem using necessary (weakly) duality;

(Step 3) solving auxiliary control and determining limiting state-average by consistency condition;

(Step 4) verifying the asymptotic social optimality of derived decentralized team strategy.\\
We now proceed step by step to construct the distributed LQG-MT strategy.
\subsection{Person-by-person optimality}\label{p-b-p optimality}

As (Step 1), we would like to propose some variational decomposition for original \eqref{LQGMT-1} around centralized strategy (although we prefer to avoid its direct computation). The person-by-person optimality principle is thus adopted for this purpose, from standpoint of a generic agent. More details are as below.

Let $\{\bar u_i\in \mathcal U_i^{c,f}\}_{i=1}^N$ be
centralized optimal team strategy (its existence can be ensured under some mild convexity conditions. But, as discussed above, such strategies are intractable for real computation purpose because of ``curse of dimensionality"). Now consider the perturbation for given benchmark agent, say, $\mathcal A_i$ use the alternative strategy $u_i\in \mathcal U_i^{c,f}$ and all other agents still apply the strategy $\bar u_{-i}=(\bar u_1,\cdots,\bar u_{i-1},\bar u_{i+1},\cdots,\bar u_N)$. The realized state \eqref{state equation} corresponding to $(u_i,\bar u_{-i})$ and $(\bar u_i,\bar u_{-i})$ are denoted by $(x_1,\cdots,x_N)$ and $(\bar x_1,\cdots,\bar x_N)$, respectively. We denote agent index set as $\mathcal{I}=\{1, \cdots, N\}$. To start the variation decomposition, it is helpful to present the following causal-relation flow-chart first:
\begin{equation*}\begin{aligned}
 &\underbrace{\delta u_i=u_i-\bar u_i}_{\text{principal basic variation}}\Longrightarrow \underbrace{\delta x_i=x_i(u_i)-\bar x_i(\bar u_i)}_{\text{principal intermediate variation}}\Longrightarrow \underbrace{\delta x_j=x_j(x_i)-\bar x_j(\bar x_i)}_{\text{marginal variation}}\\
 &\Longrightarrow \underbrace{\delta\mathcal J_j(\d u_i)}_{\text{marginal cost variation}}=\mathcal J_j(u_i,\bar u_{-i})-\mathcal J_j(\bar u_i,\bar u_{-i}), j=1,\cdots,N,\\
&\Longrightarrow  \underbrace{\delta \mathcal J_{soc}^{(N)}(\delta u_i)}_{\text{total cost variation}}=\mathcal J_{soc}^{(N)}(u_i,\bar u_{-i})-\mathcal J_{soc}^{(N)}(\bar u_i,\bar u_{-i}),
\end{aligned}\end{equation*}
where $\delta u_i$ is the most basic variation ``block" for other variation structures; we write $x_i(u_i)$ to emphasize its dependence of $x_i$ on $u_i$, and similar for $\bar x_i(\bar u_i)$; we call $\delta x_i$ the \emph{principal intermediate} variation as it depends indirectly on basic $\delta u_i$ via principal state; similarly, $\delta x_j$ \emph{marginal variations} from point of $\mathcal{A}_{i};$ also $x_j(x_i)$ depends on $x_i$ via weak-coupling $x^{(N)},$ similar to $\bar x_j(\bar x_i)$. Moreover, from standpoint of $\mathcal{A}_{i},$ the variational equations for principal state $x_i$, and marginal states $\{x_{j}\}_{j \neq i}$ satisfy:
\begin{equation}\label{variation-i}\begin{aligned}
d\delta x_i=[A_{\Theta_i}\delta x_i+B\delta u_i+F\delta x^{(N)}]dt+[C\delta x_i+D_{\Theta_i}\delta u_i+\widetilde F\delta x^{(N)}]dW_i,\qquad\delta x_i(0)=0,
\end{aligned}\end{equation}
\begin{equation}\label{minor-j}\begin{aligned}
j \neq i, \quad d\delta x_j=[A_{\Theta_j}\delta x_j+F\delta x^{(N)}]dt+[C\delta x_j+\widetilde F\delta x^{(N)}]dW_j,\
\delta x_j(0)=0.
\end{aligned}\end{equation}
Denote $\delta x_{-i}=\sum_{j\neq i}\delta x_j$ the aggregate variation of marginal agents (benchmark to $\mathcal{A}_{i}$), so applying linear state-aggregation,
\begin{equation}\label{deltax-i}\begin{aligned}
d\delta x_{-i}=[\sum_{j\neq i}A_{\Theta_j}\delta x_{j}+(N-1)F\delta x^{(N)}]dt+\sum_{j\neq i}[C\delta x_j+\widetilde F\delta x^{(N)}]dW_j,\ \delta x_{-i}(0)=0.
\end{aligned}\end{equation}
Similarly, we can also obtain the variation of cost functionals as follows. For principal cost of $\mathcal{A}_{i}$:
 \begin{equation*}\begin{aligned}
\delta \mathcal J_{i}(\d u_i)=\mathbb E\int_0^T\Big[\langle Q(\bar x_i- H\bar x^{(N)}),\delta x_i- H\delta x^{(N)}\rangle+\langle R\bar u_i,\delta u_i\rangle\Big] dt.
\end{aligned}\end{equation*}
For marginal costs of $\mathcal{A}_{i}$:
 \begin{equation*}\begin{aligned}
\delta \mathcal J_{j}(\d u_i)=\mathbb E\int_0^T\langle Q(\bar x_j- H\bar x^{(N)}),\delta x_j- H\delta x^{(N)}\rangle dt,\qquad j\neq i.
\end{aligned}\end{equation*}
Therefore, the total variation of social cost, from the person-by-person variation of $\mathcal{A}_{i}$ side, becomes
\begin{equation*}\begin{aligned}
\delta \mathcal J_{soc}^{(N)}(\d u_i)=\mathbb E\int_0^T\Big[\sum_{j=1}^N\langle Q(\bar x_j- H\bar x^{(N)}),\delta x_j- H\delta x^{(N)}\rangle+\langle R\bar u_i,\delta u_i\rangle \Big]dt.
\end{aligned}\end{equation*}
We thus have the following variation decomposition on social cost differential:
 \begin{equation}\begin{aligned}\label{cost variation-social cost-2}
\delta \mathcal J_{soc}^{(N)}(\delta u_{i})
% =&\mathbb E\int_0^T\Big[\langle Q\bar x_i,\delta x_i\rangle-\langle QH\bar x^{(N)},\delta x_i\rangle-\langle (HQ-HQH)\bar x^{(N)}, N\delta x^{(N)}\rangle\\
%  &-\sum_{j\neq i}\langle QH\bar x^{(N)},\delta x_j\rangle+\sum_{j\neq i}\langle Q\bar x_j,\delta x_j\rangle+\langle R\bar u_i,\delta u_i\rangle\Big] dt\\
 =&\mathbb E\int_0^T\Big[\langle Q\bar x_i,\delta x_i\rangle-\langle (QH+HQ-HQH)\bar x^{(N)}, \delta x_i\rangle\\
  &-\langle (QH+HQ-HQH)\bar x^{(N)}, \sum_{j\neq i}\delta x_j\rangle+\sum_{j\neq i}\langle Q\bar x_j,\delta x_j\rangle+\langle R\bar u_i,\delta u_i\rangle\Big] dt\\
=:&I_1+I_2+I_3+I_4+I_5.\\
\end{aligned}\end{equation}
There arise five decomposition terms in \eqref{cost variation-social cost-2}. Among them, $I_5$ depends directly on the \emph{principal basic variation} $\delta u_i,$  whereas $I_1, I_2$ depend on \emph{principal intermediate variation} $\delta x_{i}$ that further depends on the basic $\delta u_i.$ Moreover, $I_3, I_4$ depend on the \emph{marginal variations} $\{\delta x_{j}\}_{j \neq i}$ that further depends on the principal ones $\delta x_{i}, \delta u_{i}.$ We denote $||\delta x_{i}||_{L^{2}}=(\mathbb E\int_0^T|\delta x_{i}|^2ds)^{1/2}$. By standard SDE estimation, $||\delta x_{i}||_{L^{2}} \leq (K+O(N^{-\frac{1}{2}}) )||\delta u_{i}||_{L^{2}}$ where $K$ is independent on $N$, and only depends on coefficients of \eqref{state equation}. Moreover, $||\delta x_{j}||_{L^{2}}= O(N^{-\frac{1}{2}}) ||\delta u_{i}||_{L^{2}}$ for $j \neq i.$ Also, we remark that in general, it is not true that $||\delta x_{i}||_{L^{2}}=O(||\delta u_{i}||_{L^{2}})$.

We aim to reformulate \eqref{cost variation-social cost-2} into some variation differential based on principal terms $\delta u_i, \delta x_i$ only and some auxiliary control problem can thus be constructed in Step 2. We may realize this objective through the following procedures.

First, we need asymptote the empirical state-average $\bar x^{(N)}$ in variations $I_2, I_3$ of \eqref{cost variation-social cost-2} by its mean-field limit using heuristic reasoning. Therefore, replacing $\bar x^{(N)}$ of $I_2, I_3$ in \eqref{cost variation-social cost-2} by state-average limit $\hat x$ (to be determined later in Step 3) will yield
\begin{equation}\label{eq15}\begin{aligned}
\delta \mathcal J_{soc}^{(N)}(\delta u_i)=&\mathbb E\int_0^T\Big[\langle Q\bar x_i,\delta x_i\rangle-\langle (QH+ H Q-HQH)\hat x,\delta x_i\rangle
-\langle (QH+ H Q\\
&-HQH)\hat x,\delta x_{-i}\rangle+
 \frac{1}{N}\sum_{j\neq i}\langle Q \bar x_j,N\delta x_{j}\rangle+\langle R\bar u_i,\delta u_i\rangle \Big]dt+\varepsilon_1\\
 =:&I_1+\widehat I_2+\widehat I_3+I_4+I_5+\varepsilon_1,
\end{aligned}\end{equation}
where
$$\varepsilon_1=\mathbb E\int_0^T\langle (QH+HQ- H Q H)(\hat x-\bar x^{(N)}),N\delta x^{(N)}\rangle dt.$$
Second, note that terms $I_1,\widehat I_2, I_5$ in \eqref{eq15} already depend on the principal variations $\delta u_i$ or $\delta x_i$. Thus, we need only analyze the limiting behavior for term $\widehat I_3$ and $I_4$. It is remarkable that $\widehat I_3, I_4$ respectively involve components: $\delta x_{-i}$ and $\frac{1}{N}\sum_{j\neq i}\langle Q \bar x_j,N\delta x_{j}\rangle$ that both depend on principal basic $\delta u_i$ in rather implicit manner.

Note that for $j\neq i$, $||\delta x_{j}||_{L^{2}}= O(N^{-\frac{1}{2}}) ||\delta u_{i}||_{L^{2}},$ so $\lim_{N \rightarrow +\infty}||\delta x_{j}||_{L^{2}}=0$. Therefore, we need introduce some limiting term $x_j^*$ to replace the re-scaled $N\delta x_j$ in rate $||x_j^*-N\delta x_j||=O(N^{-\frac{1}{2}})||\delta u_{i}||_{L^{2}}$. This helps us to deal with variation of $I_4$. In addition, we introduce limiting term $x^{**}=\int_\mathcal{S} x_\theta^{**}d\Phi(\theta)$ to replace $\delta x_{-i}$ in rate that $||x^{**}-\delta x_{-i}||=O(N^{-\frac{1}{2}})||\delta u_{i}||_{L^{2}}.$ This will help us to deal with variation $\widehat I_3.$  Moreover, by the independence between $\{\Theta_j\}, \{W_j\}$ and heuristic mean-field arguments, we construct the following coupled limiting system:

\begin{equation}\label{limit process of variation}\left\{\begin{aligned}
&dx_j^*=[A_{\Theta_j}x_j^*+F\delta x_i+F\int_\mathcal{S} x_\theta^{**}d\Phi(\theta)]dt+[Cx_j^*+\widetilde F \delta x_i+\widetilde F \int_\mathcal{S} x_\theta^{**}d\Phi(\theta)]dW_j,\\
&dx_\theta^{**}=[ A_\theta x_\theta^{**}+F\delta x_i+Fx_\theta^{**}]dt,\quad x_\theta^{**}(0)=0,\\
& x_j^*(0)=0,\quad j\neq i,\quad\theta\in\mathcal{S}.
\end{aligned}\right.\end{equation}
Therefore,
\begin{equation}\label{variation-conti-1}\begin{aligned}
\delta \mathcal J_{soc}^{(N)}(\delta u_{i})
  =&\mathbb E\int_0^T\Big[\langle Q\bar x_i,\delta x_i\rangle-\langle (QH+ H Q-HQH)\hat x,\delta x_i\rangle
-\langle (QH+ H Q\\
&-HQH)\hat x,x^{**}\rangle+
 \frac{1}{N}\sum_{j\neq i}\langle Q\bar x_j, x_j^*\rangle+\langle R\bar u_i,\delta u_i\rangle\Big] dt+\sum_{l=1}^3\varepsilon_l\\
 =:&I_1+\widehat I_2+\widetilde I_3+\widetilde I_4+I_5+\sum_{l=1}^3\varepsilon_l,
\end{aligned}\end{equation}
where
\begin{equation*}\left\{\begin{aligned}
&\varepsilon_2= \mathbb{E}\int_0^T\langle (QH+ H Q-HQH)\hat x,x^{**}-\delta x_{-i}\rangle dt,\\
&\varepsilon_3=\mathbb E\int_0^T\frac{1}{N}\sum_{j\neq i}\langle Q\bar x_j,N\delta x_j-x^*_j\rangle dt.
\end{aligned}\right.\end{equation*}
Noting $\widetilde{I}_4$ of \eqref{variation-conti-1} connects to a sequence of exchangeable random variables  $\{\int_0^T\langle Q\bar x_j, x_j^*\rangle dt\}\\ \in L^{1}_{\mathcal F_T}(\Omega;\mathbb R)$.  By de Finetti theorem, they are \emph{conditionally} independent identically distributed with respect to some tail sigma-algebra. Also, it is observable that such tail sigma-algebra should depend on $\delta x_{i}$ in rather implicit way. Then, we may apply conditional law of large number to identify the related average. We present some weak duality approach to break away $\delta \mathcal J_{soc}^{(N)}(\d u_i)$ from dependence on $x_j^*$ and $x^{**}$.

\subsection{Weak construction duality}
In order to break away $\delta \mathcal J_{soc}^{(N)}(\d u_i)$ of \eqref{variation-conti-1} from direct dependence on $x_j^*$ and $x^{**}$ (see $\widetilde{I}_3,\widetilde{I}_4$),
 we introduce the following adjoint equations $\{y_1^j\}_{j\neq i}$ and $y_2^\theta$ satisfying:
\begin{equation}\label{adjoint processes}\left\{\begin{aligned}
&dy_1^j=\alpha_1^j dt+\beta_{1}^{jj}dW_j+\sum_{l=1,l\neq j}^N\beta_{1}^{jl}dW_l,\qquad y_1^j(T)=0,\quad j\neq i,\\
&dy_2^\theta=\alpha_2^\theta dt,\qquad y_2^\theta(T)=0,\quad \theta\in\mathcal{S},
\end{aligned}\right.\end{equation}where $\{W_l\}_{l \neq i}$ are some Brownian motion copies matching all marginal agents in large-scaled system, from the benchmark point of $\mathcal{A}_{i}.$ We remark that $y_2^\theta$ is parameterized by diversity index in continuum support: $\theta \in \mathcal{S},$ while $y_1^j$ is parameterized by marginal agent index $j \neq i \in \mathcal{I}.$ Accordingly, the duality below should be some weak construction in \emph{distributional} and \emph{agent-wise} sense, respectively indexed by $\theta \in \mathcal{S}$ and $j \in \mathcal{I}.$ To start, first apply It\^{o}'s formula to $\langle y_1^j,x_j^*\rangle$ for each marginal agent index $j \neq i,$
%\begin{equation*}\begin{aligned}
%d\langle y_1^j,x_j^*\rangle=[\langle y_1^j,A_{\Theta_j}x_j^*+F\delta x_i+Fx^{**}\rangle+\langle \alpha_1^j,x_j^*\rangle+
%\langle\beta_{1}^{jj},Cx_j^*+\widetilde F\delta x_i+\widetilde Fx^{**}\rangle]dt+\sum_{j=1}^N(\cdots)dW_j(t).
%\end{aligned}\end{equation*}
integrating from $0$ to $T$ and taking expectation, by countable agent-wise addition for all $j \in \mathcal{I}\backslash i$,
\begin{equation}\label{new-22}\begin{aligned}
0=&\mathbb E\int_0^T\Big[\frac{1}{N}\sum_{j\neq i}\langle\alpha_1^j+A_{\Theta_j}^\top y_1^j+C^\top \beta_{1}^{jj},x_j^*\rangle+\frac{1}{N}\sum_{j\neq i}\langle F^\top y_1^j
+\widetilde F^\top\beta_{1}^{jj},x^{**}\rangle\\
&+\frac{1}{N}\sum_{j\neq i}
\langle F^\top y_1^j+\widetilde F^\top\beta_{1}^{jj},\delta x_i\rangle\Big]dt.
\end{aligned}\end{equation}
Similarly, by distributed integral on all $\theta \in \mathcal{S},$
\begin{equation}\label{new-25}\begin{aligned}
0=\int_0^T\Big[\int_\mathcal{S}\langle\alpha_2^\theta+ A_\theta^\top  y_2^\theta+F^\top y_2^\theta,x_\theta^{**}\rangle d\Phi(\theta)+\int_\mathcal{S}\langle F^\top y_2^\theta,\delta x_i\rangle d\Phi(\theta)\Big]dt.
\end{aligned}\end{equation}
Combing \eqref{new-22} and \eqref{new-25} with \eqref{variation-conti-1}
\begin{equation}\label{new-26}\begin{aligned}
&\delta \mathcal J_{soc}^{(N)}(\delta u_{i})
=\mathbb E\int_0^T\Big[\langle Q\bar x_i,\delta x_i\rangle-\langle (QH+ H Q-HQH)\hat x,\delta x_i\rangle-\frac{1}{N}\sum_{j\neq i}
\langle F^\top y_1^j+\widetilde F^\top\beta_{1}^{jj},\delta x_i\rangle\\
&-\int_\mathcal{S}\langle F^\top y_2^\theta,\delta x_i\rangle d\Phi(\theta)+\langle R\bar u_i,\delta u_i\rangle\Big] dt+\mathbb E\int_0^T\Big[\frac{1}{N}\sum_{j\neq i}\langle Q\bar x_j-\alpha_1^j-A_{\Theta_j}^\top y_1^j-C^\top \beta_{1}^{jj}, x_j^*\rangle\Big]dt\\
&-\mathbb E\int_0^T\int_{\mathcal S}\langle (QH+ H Q-HQH)\hat x+\frac{1}{N}\sum_{j\neq i}(F^\top y_1^j
+\widetilde F^\top\beta_{1}^{jj})+\alpha_2^\theta+ A_\theta^\top  y_2^\theta+F^\top y_2^\theta ,x_\theta^{**}\rangle d\Phi(\theta)dt\\
&+\sum_{l=1}^3\varepsilon_l.\\
\end{aligned}\end{equation}
Let
\begin{equation*}\left\{\begin{aligned}
\alpha_1^j=&Q\bar x_j-A_{\Theta_j}^\top y_1^j-C^\top\beta_{1}^{jj},\\
\alpha_2^\theta
=&-(QH+ H Q-HQH)\hat x-F^\top\mathbb E y_1^j
-\widetilde F^\top\mathbb E\beta_{1}^{jj}-A_\theta^\top y_2^\theta -F^\top  y_2^\theta,
\end{aligned}\right.\end{equation*}
hence we reach the following weak duality adjoint process:
\begin{equation}\label{explicit adjoint}\left\{\begin{aligned}
&dy_1^j=\left(Q\bar x_j-A_{\Theta_j}^\top y_1^j-C^\top\beta_{1}^{jj}\right) dt+\beta_{1}^{jj}dW_j+\sum_{l=1,l\neq j}^N\beta_{1}^{jl}dW_l,\\
&dy_2^\theta=-\Big((QH+ H Q-HQH)\hat x-(F^\top \mathbb Ey_1^j+\widetilde F^\top\mathbb E \beta_1^{jj})-A_\theta^\top y_2^\theta -F^\top  y_2^\theta\Big)dt,\\
& y_1^j(T)=0,\ j\neq i,\ y_2^\theta(T)=0,\ \theta\in\mathcal{S}.
\end{aligned}\right.\end{equation}
We point out that above system can be rewritten as
\begin{equation*}\left\{\begin{aligned}
&dy_1^j=\left(Q\bar x_j-A_{\Theta_j}^\top y_1^j-C^\top\beta_{1}^{jj}\right) dt+\beta_{1}^{jj}dW_j+\sum_{l=1,l\neq j}^N\beta_{1}^{jl}dW_l,\\
&dy_2^{\Theta}=-\left((QH+ H Q-HQH)\hat x-(F^\top \mathbb Ey_1^j+\widetilde F^\top\mathbb E \beta_1^{jj})-A^{\top}_\Theta y_2^\Theta -F^\top y_2^{\Theta}\right)dt,\\
& y_1^j(T)=0,\ j\neq i,\ y_2^\Theta(T)=0.
\end{aligned}\right.\end{equation*}
We remark that $y_2^{\Theta}$ is a degenerate BSDE by noting $\Theta \in \mathcal{F}_{0}$. Also, it is not necessary to specify any dependence assumption between $\Theta_{j}$ and $\Theta$ since $y_{1}^{j}$ and $y_2^\Theta$ get coupled only through expectation operator. In other words, their coupling here and further variant in consistency condition, only depend on the expectation in distribution sense. Again, this is why we term the resultant duality as weak-duality.
Substituting \eqref{explicit adjoint} into \eqref{new-26}, we have
\begin{equation*}\begin{aligned}
\delta \mathcal J_{soc}^{(N)}(\delta u_{i})
  =&\mathbb E\int_0^T\Big[\langle Q\bar x_i,\delta x_i\rangle-\langle (QH+ H Q-HQH)\hat x,\delta x_i\rangle-\frac{1}{N}\sum_{j\neq i}
\langle F^\top y_1^j+\widetilde F^\top\beta_{1}^{jj},\delta x_i\rangle\\
&-\int_\mathcal{S}\langle F^\top y_2^\theta,\delta x_i\rangle d\Phi(\theta)+\langle R\bar u_i,\delta u_i\rangle\Big] dt+\sum_{l=1}^4\varepsilon_l,\\
\end{aligned}\end{equation*}
where
\begin{equation*}\begin{aligned}\label{}
&{\varepsilon}_{4}=\mathbb E\int_0^T\langle F^\top(\mathbb E[y_1^j]-\frac{1}{N}\sum_{j\neq i}y_1^j)+\widetilde F^\top(\mathbb E[\beta_1^{jj}]-\frac{1}{N}\sum_{j\neq i}\beta_{1}^{jj}) ,x^{**}\rangle dt.\\
\end{aligned}\end{equation*}
We observe that the initial terms such as $\langle Q\bar x_j, x_j^*\rangle$ in \eqref{variation-conti-1}, is now reformulated  with some inner product between principal intermediate variation $\delta x_{i}$ and some quantities in terms by $y_{2}^{\theta}$ and $y_{1}^{j}$ in an agent-wise (i.e., $j \neq i$) manner. Then, we can identify the tail filtration for exchangeable $\{\int_0^T\langle Q\bar x_j, x_j^*\rangle dt\}_{j \neq i}$ based on $\delta x_{i}$ with a degenerated filtration. So, applying conditional law of large number, and noticing $\{y_1^j$, $j\neq i\}$ are identical distributed, we reach the following representation with expectation operator:
\begin{equation}\label{variation-conti-2}\begin{aligned}
\delta \mathcal J_{soc}^{(N)}(\delta u_{i})
 =&\mathbb E\int_0^T\Big[\langle Q\bar x_i,\delta x_i\rangle-\langle (QH+ H Q-HQH)\hat x+F^\top \mathbb{E}[y_1]+\widetilde F^\top\mathbb E[\beta_1^{1}]\\
 &+F^\top \int_\mathcal{S} y_2^\theta d\Phi(\theta),\delta x_i\rangle
+\langle R\bar u_i,\delta u_i\rangle \Big] dt+\sum_{l=1}^{5}\varepsilon_l,\\
\end{aligned}\end{equation}
where $y_1$ (depending on $\bar x_1$, that is the optimized state for generic agent) is some copy with same distribution for generic $y_{1}^{j}$:
\begin{equation}\left\{\begin{aligned}\label{y1,y2theta}
&dy_1=[Q\bar x_1-A_{\Theta}^\top y_1-C^\top\beta_{1}^{1}] dt+\beta_{1}^{1}dW_1+\sum_{l=1,l\neq 1}^N\beta_{1}^{l}dW_l,\\
&dy_2^\theta=[-(QH+ H Q-HQH)\hat x-(F^\top \mathbb Ey_1+\widetilde F^\top\mathbb E \beta_{1}^{1})- A_\theta^\top y_2^\theta-F^\top y_2^\theta ]dt,\\
 &y_1(T)=0,\qquad y_2^\theta(T)=0,\quad\theta\in\mathcal{S},
\end{aligned}\right.\end{equation}
and
\begin{equation*}\begin{aligned}\label{}
&\varepsilon_{5}=\mathbb E\int_0^T\langle F^\top(\mathbb E[y_1]-\frac{1}{N}\sum_{j\neq i}y_1^j)+\widetilde F^\top(\mathbb E[\beta_1^{1}]-\frac{1}{N}\sum_{j\neq i}\beta_{1}^{jj}) ,\delta x_i\rangle dt.\\
\end{aligned}\end{equation*}
We remark that $y_1$ has the same distribution with generic $y_{1}^{j}$, thus we call above procedure as weak duality.
We point out all variations terms in \eqref{variation-conti-2}, are now directly depending only on principal (basic, or intermediate) variations. Thus, we now formulate a decentralized auxiliary cost differential $\delta J_i(\delta u_{i})$:
\begin{equation}\label{variation-conti-3}\begin{aligned}
\delta J_i(\delta u_{i})
 =\mathbb E\int_0^T\Big[&\langle Q\bar x_i,\delta x_i\rangle-\langle (QH+ H Q-HQH)\hat x+F^\top \hat{y}_1+\widetilde F^\top\hat{\beta}_1\\
 &+F^\top \int_\mathcal{S} y_2^\theta d\Phi(\theta),\delta x_i\rangle+\langle R\bar u_i,\delta u_i\rangle\Big] dt.
\end{aligned}\end{equation}

\begin{remark}
There are four undetermined terms in \eqref{variation-conti-3} respectively \emph{:} $\hat x$ by \eqref{eq15} is the state-average limit; $(\hat y_1=\mathbb E[y_1],\hat \beta_1=\mathbb E[\beta_1^{1}],y_2^\theta)$ is from \eqref{y1,y2theta} because of the weak construction duality procedure. All these terms, especially $\hat x$, will be determined by consistency condition (CC) in Section \ref{decentralized strategy}.
\end{remark}
\begin{remark}
In \eqref{variation-conti-3}, we introduce the
first variation of auxiliary cost functional $\delta J_i(\d u_i)$ and ignore the error term $\varepsilon_l$, $l=1,\cdots,5$. The convergence rate estimation of these terms and the rigorous
proofs will be given in Section \ref{asymptotic optimality}.
\end{remark}

\section{Auxiliary control problem and consistency condition}\label{decentralized strategy}
\subsection{Auxiliary control with double-projection}
By \eqref{variation-conti-3}, we can introduce the following (auxiliary control (AC) problem) for a generic $\mathcal{A}_{i}$:\\
\begin{equation*}\text{(AC):}\left\{
\begin{aligned}
&\text{ Minimize }
\begin{aligned}
J_i(u_i(\cdot))
 =&\frac{1}{2}\mathbb E\int_0^T\Big[\langle Q x_i, x_i\rangle-2\langle \Xi, x_i\rangle+\langle R u_i, u_i\rangle\Big] dt,
      \end{aligned}\\
&\text{ subject to }
\begin{aligned}
dx_i(t)=[A_{\Theta_i}x_i+Bu_i+F\hat x]dt+[Cx_i+D_{\Theta_i}u_i+\widetilde  F\hat x]dW_i(t), x_i(0)=\xi,
\end{aligned}
\end{aligned}\right.
\end{equation*}
with
\begin{equation}\label{undetermined}
\Xi(t; \hat{x}, y_2^\theta,\hat y_1,\hat\beta_1)=(QH+ H Q-HQH)\hat x+F^\top \hat{y}_1+\widetilde F^\top\hat{\beta}_1+F^\top \int_\mathcal{S} y_2^\theta d\Phi(\theta),
\end{equation}
where $\hat x$ is the limiting state-average term introduced in \eqref{eq15}; $(y_2^\theta,\hat y_1,\hat\beta_1)$ depends on $\hat x$ satisfying dynamics \eqref{y1,y2theta}. Also, we remark that $\hat y_1$ depends on optimal state $\bar{x}_{j}.$

We will apply stochastic maximum principle to study \textbf{Problem (AC)}. To this end, we introduce the following first-order adjoint equation:
\begin{equation*}\begin{aligned}
dp_i(t)=-[A_{\Theta_i}^\top p_i+Qx_i-\Xi+C^\top q_i]dt+q_idW_i(t),\qquad p_i(T)=0.
\end{aligned}\end{equation*}
Let $u_i^*$ be the optimal control and $(x_i^*,p_i^*,q_i^*)$ the corresponding state and adjoint state. For any $u_i\in L^2_{\mathbb G^i}(0,T;\mathbb R^m)$ such that  $u_i^*+u_i\in\mathcal U_{i,op}^{d,p}$, we have $u_i^\epsilon:=u_i^*+\epsilon u_i\in\mathcal U_{i,op}^{d,p}$. The corresponding state and adjoint state with respect to $u_i^\epsilon$ are denoted by $(x_i^\epsilon,p_i^\epsilon,q_i^\epsilon)$. Introduce the following variational equation
\begin{equation*}\begin{aligned}
dy_i(t)=[A_{\Theta_i}y_i+Bu_i]dt+[Cy_i+D_{\Theta_i}u_i]dW_i(t),\qquad y_i(0)=0.
\end{aligned}\end{equation*}
Applying It\^{o}'s formula to $\langle p_i,y_i\rangle$, by the optimality of $u^*_i$ (i.e., $J_i(u_i^\epsilon)-J_i(u_i^*)\geq0$), we have
$$\mathbb E\int_0^T\langle Ru_i^*+B^\top p_i+D_{\Theta_i}^\top q_i,u_i\rangle ds\geq0.$$
For any $0\leq t\leq T$ and $\mathcal G^i_t$-measurable random variable $\eta_i$,
let
\begin{equation*}u_i^*(s)+u_i(s)=\left\{\begin{aligned}&u_i^*(s),\quad &s\notin[t,t+\epsilon];\\
&\eta_i,&s\in[t,t+\epsilon].\end{aligned}\right.\end{equation*}
Therefore,
$$\frac{1}{\epsilon}\mathbb E\int_t^{t+\epsilon}\langle Ru_i^*+B^\top p_i+D_{\Theta_i}^\top q_i,\eta_i-u_i^*\rangle ds\geq0.$$
Let $\epsilon\rightarrow0$, we have
$$\mathbb E\langle R(t)u_i^*(t)+B^\top(t)p_i^*(t)+D_{\Theta_i}^\top(t)q_i^*(t),\eta_i-u^*_i(t)\rangle\geq0,\qquad t\in[0,T].$$
For any $v\in\Gamma$ and $A\in\mathcal G^i_t$, define $\eta_i=vI_A+u_i^*(t)I_{A^c}$, we have
$$\mathbb E\langle R(t)u_i^*(t)+B^\top(t)p_i^*(t)+D_{\Theta_i}^\top(t)q_i^*(t),v-u^*_i(t)\rangle I_A\geq0,\qquad t\in[0,T].$$
Since $A\in\mathcal G^i_t$ is arbitrary, we have
$$\mathbb E[\langle R(t)u_i^*(t)+B^\top(t)p_i^*(t)+D_{\Theta_i}^\top(t)q_i^*(t),v-u^*_i(t)\rangle|\mathcal G^i_t]\geq0,\qquad t\in[0,T],\ \mathbb P-a.s.$$
i.e.,
\begin{equation}\label{SMP}\langle -R(t)u_i^*(t)+\mathbb E[-B^\top(t)p_i^*(t)-D_{\Theta_i}^\top(t)q_i^*(t)|\mathcal G^i_t],v-u^*_i(t)\rangle\leq0,\ t\in[0,T],\ \mathbb P-a.s.
\end{equation}
Since $v\in\Gamma$ is arbitrary and $\Gamma$ is a closed convex set, it follows from the well-known results of convex analysis that \eqref{SMP} is equivalent to
\begin{equation}\label{optimal control}
u_i^*(t)=\mathbf P_{\Gamma}[R^{-1}\mathbb E[-B^\top p_i^*(t)-D_{\Theta_i}^\top q_i^*(t)|\mathcal G^i_t]],\ a.e.\ t\in[0,T],\ \mathbb P-a.s.,
\end{equation}
where $\mathbf P_{\Gamma}[\cdot]$ is the projection mapping from $\mathbb R^m$ to its closed convex subset $\Gamma$ under the norm $\|v\|_{R}^2:=\langle R^{\frac{1}{2}}v,R^{\frac{1}{2}}v\rangle$. We point out that there involves two projections in \eqref{optimal control}, because of the input constraint and partial information constraint. This differs from \cite{HHL2017,HHN2018} which include only single-projection on input set. Furthermore, the two projections are non-commutative due to above maximum principle arguments.
In this case, the related Hamiltonian system for (AC) problem becomes
\begin{equation}\label{Hamiltonian system}\left\{\begin{aligned}
dx_i^*=&\Big[A_{\Theta_i}x_i^*+B\mathbf P_{\Gamma}[R^{-1}\mathbb E[-B^\top p_i^*(t)-D_{\Theta_i}^\top q_i^*(t)|\mathcal G^i_t]]+F\hat x\Big]dt\\
&+\Big[Cx_i^*+D_{\Theta_i}\mathbf P_{\Gamma}[R^{-1}\mathbb E[-B^\top p_i^*(t)-D_{\Theta_i}^\top q_i^*(t)|\mathcal G^i_t]]+\widetilde  F\hat x\Big]dW_i(t),\\
dp_i^*=&-[A_{\Theta_i}^\top p_i^*+Qx_i^*-\Xi+C^\top q_i^*]dt+q_i^*dW_i(t),\\
 x_i^*(0)=&\xi,\quad p_i^*(T)=0,
\end{aligned}\right.\end{equation}
which is a fully-coupled FBSDEs with double-projection: the mapping on input convex-closed set, and the filtering for partial information (i.e., conditional expectation on sub-space).

\subsection{Consistency condition}\label{Consistency condition}

In this section, we will characterize the undetermined processes, especially state-average limit $\hat x$, in \eqref{undetermined} via some consistency matching scheme. Given the Hamiltonian system by \eqref{Hamiltonian system},  all agents should apply some exchangeable team decisions $\{{u}_{i}^*\}_{i=1}^{N}$ and the realized states should be as follows: \begin{equation*}\label{}\left\{\begin{aligned}
dx_i^*=&\Big[A_{\Theta_i}x_i^*+B\mathbf P_{\Gamma}[R^{-1}\mathbb E[-B^\top p_i^*(t)-D_{\Theta_i}^\top q_i^*(t)|\mathcal G^i_t]]+Fx^{*,(N)}\Big]dt\\
&+\Big[Cx_i^*+D_{\Theta_i}\mathbf P_{\Gamma}[R^{-1}\mathbb E[-B^\top p_i^*(t)-D_{\Theta_i}^\top q_i^*(t)|\mathcal G^i_t]]+\widetilde  Fx^{*,(N)}\Big]dW_i(t),\\
%dp_i^*=&-[A_{\Theta_i}^\top p_i^*+Qx_i^*-\Xi+C^\top q_i^*]dt+q_i^*dW_i(t),\\
 x_i^*(0)=&\xi,
\end{aligned}\right.\end{equation*}
where $x^{*,(N)}=\frac{1}{N}\sum_{i=1}^Nx_i^*$ and $(p_i^*,q_i^*)$ is the solution of \eqref{Hamiltonian system}. Making all such exchangeable strategies aggregated, and applying de Finetti theorem,  we can obtain the limiting system by identifying $\hat{x}=\mathbb{E}x^{*}$,
\begin{equation}\label{limiting Hamilton}\left\{\begin{aligned}
d\widetilde x=&\Big[A_{\Theta}\widetilde x+B\mathbf P_{\Gamma}[R^{-1}\mathbb E[-B^\top\widetilde  p(t)-D_{\Theta}^\top\widetilde  q(t)|\mathcal G_t]]+F\mathbb E \widetilde x\Big]dt\\
&+\Big[C\widetilde x+D_{\Theta}\mathbf P_{\Gamma}[R^{-1}\mathbb E[-B^\top \widetilde p(t)-D_{\Theta}^\top \widetilde q(t)|\mathcal G_t]]+\widetilde  F\mathbb E \widetilde x\Big]dW(t),\\
d\widetilde p=&-\Big[A_{\Theta}^\top \widetilde p+Q\widetilde x-
(QH+ H Q-HQH)\mathbb E\widetilde x-F^\top \hat{y}_1-\widetilde F^\top\hat{\beta}_1
\\
&-F^\top \int_\mathcal{S} y_2^\theta d\Phi(\theta)+C^\top \widetilde q\Big]dt+\widetilde qdW(t),\\
 \widetilde x(0)=&\xi,\quad \widetilde p(T)=0,
\end{aligned}\right.\end{equation}
where $\Theta$ is a random variable with distribution defined in (A1), $W(t)$ is a generic Brownian motion independent of $\Theta$, $\mathcal G$ is sub-filtration representing the partial information and $(\hat y_1=\mathbb E[y_1],\hat \beta_1=\mathbb E[\beta_1^{1}],y_2^\theta)$ is from \eqref{y1,y2theta}.
Note that we suppress subscript $i$ in \eqref{limiting Hamilton} as all agents are statistically identical in the distribution sense. Combing with
\eqref{y1,y2theta}, we will obtain consistency condition (CC) of \textbf{Problem LQG-MT}. For simplicity, define
$$\mathcal E_t[-B^\top \gamma-D_{\Theta}^\top \vartheta]=\mathbb E[-B^\top \gamma-D_{\Theta}^\top \vartheta|\mathcal G_t].$$
Hence we have the following result.
 \begin{theorem}\label{CC of LQGMT}
The undetermined parameters of \eqref{undetermined} can be determined by
$$(\hat x,\hat y_1,\hat\beta_1, y_2^\theta)=(\mathbb E \alpha,\mathbb E\check y_1,\mathbb E\check\beta_1,\check y_2^\theta),$$ where $(\alpha,\gamma,\vartheta,\check y_1,\check\beta_1,\check y_2^\theta)$ is the solution of the consistency condition of \textbf{\emph{Problem LQG-MT}}:
\begin{equation}\label{CC}\left\{\begin{aligned}
&d\alpha=[A_\Theta\alpha+B\mathbf P_{\Gamma}[R^{-1}\mathcal E_t[-B^\top \gamma-D_{\Theta}^\top \vartheta]]+F\mathbb E\alpha]dt\\
&\qquad+[C\alpha+D_\Theta\mathbf P_{\Gamma}[R^{-1}\mathcal E_t[-B^\top \gamma-D_{\Theta}^\top \vartheta]]+\widetilde F\mathbb E\alpha]dW,\\
&d\gamma=[-Q\alpha+(QH+ H Q-HQH)\mathbb E\alpha-A_\Theta^\top \gamma+F^\top \int_\mathcal{S}\check y_2^\theta d\Phi(\theta)+F^\top \mathbb E\check{y}_1\\
&\qquad-C^\top \vartheta+\widetilde F^\top\mathbb E\check{\beta}_1]dt+\vartheta dW,\\
&d\check y_1=[Q\alpha-A_\Theta^\top \check y_1-C^\top\check \beta_1] dt+\check \beta_1dW,\\
&d\check y_2^\theta=[-(QH+ H Q-HQH)\mathbb E\alpha-F^\top \mathbb E\check y_1-\widetilde F^\top\mathbb E \check \beta_1- A_\theta^\top \check y_2^\theta-F^\top \check y_2^\theta]dt,\\
 &
\alpha(0)=\xi,\quad \gamma(T)=0,\quad \check y_1(T)=0,\quad \check y_2^\theta(T)=0,\quad\theta\in\mathcal{S}.
\end{aligned}\right.\end{equation}
\end{theorem}
\begin{remark}
\eqref{CC} is a new type of fully-couple FBSDEs with double-projection (projection mapping on the convex-closed sub-set and partial-information sub-space). Moreover, both temporal variable $t$ and spatial variable $\theta$ appear in \eqref{CC}. Considering this, we can rewrite \eqref{CC} in the following more compact form:
\begin{equation}\label{product space}\left\{\begin{aligned}
&d\alpha=[A_\Theta\alpha+B\mathbf P_{\Gamma}[R^{-1}\mathcal E_t[-B^\top \gamma-D_{\Theta}^\top \vartheta]]+F\mathbb E\alpha]dt\\
&\qquad+[C\alpha+D_\Theta\mathbf P_{\Gamma}[R^{-1}\mathcal E_t[-B^\top \gamma-D_{\Theta}^\top \vartheta]]+\widetilde F\mathbb E\alpha]dW,\\
&d\gamma=[-Q\alpha+(QH+ H Q-HQH)\mathbb E\alpha-A_\Theta^\top \gamma+F^\top \mathbb E\check y_2^\Theta +F^\top \mathbb E\check{y}_1-C^\top \vartheta\\
&\qquad+\widetilde F^\top\mathbb E\check{\beta}_1]dt+\vartheta dW(t),\\
&d\check y_1=[Q\alpha-A_\Theta^\top \check y_1-C^\top\check \beta_1] dt+\check \beta_1dW,\\
&d\check y_2^\Theta=[-(QH+ H Q-HQH)\mathbb E\alpha-F^\top \mathbb E\check y_1-\widetilde F^\top\mathbb E \check \beta_1- A_\Theta^\top \check y_2^\Theta-F^\top\check y_2^\Theta ]dt,\\
 &
\alpha(0)=\xi,\quad \gamma(T)=0,\quad \check y_1(T)=0,\quad \check y_2^\Theta(T)=0.
\end{aligned}\right.\end{equation}
Note that by the independence between $\Theta$ and $W$, \eqref{product space} can be viewed as defined on the product space $\Omega_1\times\Omega_2\rightarrow\mathcal{S}\times\mathbb R^n$. This is a general system which includes many framework in current literature as special cases. For more information, please refer to Section \ref{comparison}.
\end{remark}
The well-posedness of \eqref{CC} will be studied in Section \ref{well-posedness of CC}.

\subsection{Literature comparison}\label{comparison}We now present comparisons to some relevant literature.

\subsubsection{Homogeneous case without diversity}
For the homogeneous case with $\mathcal{S}=\{s_1\}$ being singleton set, we have $A_{\Theta_i}=A_{s_1}:=A$ and $D_{\Theta_i}=D_{s_1}:=D$ for $i=1,\cdots,N$. In this case, we do not need to introduce $x_\theta^{**}$ as in \eqref{limit process of variation} when applying person-by-person optimality. We only need to introduce $x^{**}$ to replace $\delta x_{-i}$. In fact, in current case, $x^{**}$ satisfies the following dynamics:
\begin{equation*}\begin{aligned}
&dx^{**}=[  (A+F)x^{**}+F\delta x_i]dt,\quad x^{**}(0)=0.
\end{aligned}\end{equation*}
Moreover, CC in homogeneous case becomes
\begin{equation}\label{CC-2}\left\{\begin{aligned}
&d\alpha=[A\alpha+B\mathbf P_{\Gamma}[R^{-1}\mathcal E_t[-B^\top \gamma-D^\top \vartheta]]+F\mathbb E\alpha]dt\\
&\qquad+[C\alpha+D\mathbf P_{\Gamma}[R^{-1}\mathcal E_t[-B^\top \gamma-D^\top \vartheta]]+\widetilde F\mathbb E\alpha]dW,\\
&d\gamma=[-Q\alpha+(QH+ H Q-HQH)\mathbb E\alpha-A^\top \gamma+F^\top \check y_2+F^\top \mathbb E\check{y}_1-C^\top \vartheta\\
&\qquad+\widetilde F^\top\mathbb E\check{\beta}_1]dt+\vartheta dW(t),\\
&d\check y_1=[Q\alpha-A^\top \check y_1-C^\top\check \beta_1] dt+\check \beta_1dW,\\
&d\check y_2=[-(QH+ H Q-HQH)\mathbb E\alpha-F^\top \mathbb E\check y_1-\widetilde F^\top\mathbb E \check \beta_1- A^\top \check y_2-F^\top\check y_2]dt,\\
 &
\alpha(0)=\xi,\quad \gamma(T)=0,\quad \check y_1(T)=0,\quad \check y_2(T)=0.
\end{aligned}\right.\end{equation}This is the special case of \eqref{CC} with $\Phi(\theta)$ being a Dirac distribution. Subsequently, our framework covers the homogeneous case as its special case. Furthermore, in case $C=D=F=\widetilde{F}=0$, $\Gamma=\mathbb{R}^{m}$ and $\mathbb{G}^i=\mathbb{F}^i$,
%$\eqref{CC-2}$ further reduces to
%\begin{equation}\label{CC-3}\left\{\begin{aligned}
%&d\alpha=[A\alpha-BR^{-1}B^\top \gamma]dt,\\
%&d\gamma=[-Q\alpha+(QH+ H Q-HQH)\mathbb E\alpha-A^\top \gamma]dt+\vartheta dW(t),\\
%&d\check y_1=[Q\alpha-A^\top \check y_1] dt+\check \beta_1dW,\\
%&d\check y_2=[-(QH+ H Q-HQH)\mathbb E\alpha- A^\top \check y_2]dt,\\
% &
%\alpha(0)=\xi,\quad \gamma(T)=0,\quad \check y_1(T)=0,\quad \check y_2(T)=0.
%\end{aligned}\right.\end{equation}
%In \eqref{CC-3}, $(\check y_1,\check\beta_1,\check y_2)$ and $(\alpha,\gamma,\vartheta)$ become decoupled.
by applying expectation, $\bar\alpha=\mathbb{E}\alpha$ and $\bar\gamma=\mathbb{E}\gamma$ satisfy the dynamics:
\begin{equation}\label{CC-4}\left\{\begin{aligned}
&d\bar\alpha=[A\bar\alpha-BR^{-1}B^\top \bar\gamma]dt,\\
&d\bar\gamma=[(-Q+QH+ H Q-HQH)\bar\alpha-A^\top \bar\gamma]dt,\\
 &\bar\alpha(0)=\xi,\quad \bar\gamma(T)=0.
\end{aligned}\right.\end{equation}
This is just the special case discussed in pp. 1742 of \cite{HCM2012} (see (42),(43) therein). The only difference is that \eqref{CC-4} is of open-loop ($\bar\gamma$ is the adjoint process) while (42) and (43) in \cite{HCM2012} are of closed-loop ($\Pi\bar x+s$ is of feedback form).

\subsubsection{Heterogeneous case with finite diversities}
Specifically, we assume that $\Theta_i$ is deterministic (post-sampling) and assumes values in a finite discrete set $\mathcal{S}=\{1,2,\cdots,K\}$.  For $1 \leq k \leq K$, introduce$$\mathcal{I}_k=\{i|\Theta_i=k, 1 \leq i \leq  N\}, \quad \quad N_k=|\mathcal{I}_k|,$$where $N_k$ is the cardinality of index set $\mathcal{I}_k$ (i.e., cardinality of set of $k$-type agents). For $1\leq k\leq K$, let $\pi_k^{(N)}=\frac{N_k}{N}$, then $\pi^{(N)}=(\pi_1^{(N)}, \cdots, \pi_K^{(N)})$ is a probability vector representing the empirical distribution of $\Theta_1, \cdots, \Theta_N.$ Suppose there exists a probability mass vector $\pi=(\pi_1, \cdots, \pi_K)$ such that $\displaystyle{\lim_{N\rightarrow+\infty}}\pi^{(N)}=\pi$ and $\displaystyle{\min_{1 \leq k \leq K}}\pi_{k}>0.$
Under these assumptions, the person-by-person procedure still proceeds as in Section \ref{p-b-p optimality}. Let $\delta x_{(k)}=\sum_{j\in\mathcal I_k,j\neq i}\delta x_j$. By exchangeability of agents within same type, we need only consider a representative agent in each type when using a limit to approximate $\delta x_{(k)}$. Therefore, for $k=1,\cdots,K$, we should introduce the term
 $x^{**}_k$ to replace $\delta x_{(k)}$, where $x^{**}_k$ satisfies the following dynamics:
\begin{equation*}\begin{aligned}
&dx_k^{**}=\Big[A_kx_k^{**}+F\pi_k\delta x_i+F\pi_k\sum_{l=1}^Kx_l^{**}\Big]dt,\qquad x_k^{**}(0)=0 ,\quad k=1,\cdots,K.
\end{aligned}\end{equation*}
%Hence, this corresponds to the case that $\mathbb P(\Theta=k)=\pi_k$, $k=1,\cdots,K$.
Furthermore, if $\mathbb{G}^i=\mathbb{F}^i$, CC of heterogeneous case with finite diversities becomes:
\begin{equation}\label{CC-hete-finite}\left\{\begin{aligned}
&d\alpha_k=[A_k\alpha_k+B\mathbf P_{\Gamma}[R_{k}^{-1}(B^\top \gamma_k+D_{k}^\top \vartheta_k)]+F\sum_{l=1}^K\pi_l\mathbb E\alpha_l]dt\\
&\qquad\quad+[C\alpha_k+D_k\mathbf P_{\Gamma}[R_{k}^{-1}(B^\top \gamma_k+D_{k}^\top \vartheta_k)]+\widetilde F\sum_{l=1}^K\pi_l\mathbb E\alpha_l]dW_k(t),\\
&d\gamma_k=[-Q\alpha_k+(QH+ H Q-HQH)\sum_{l=1}^K\pi_l\mathbb E\alpha_l-A_k^\top \gamma_k+F^\top\sum_{l=1}^K\pi_l \check y_2^l+F^\top\sum_{l=1}^K\pi_l \mathbb E\check{y}_1^l\\
&\qquad\quad-C^\top \vartheta_k+\widetilde F^\top\sum_{l=1}^K\pi_l\mathbb E\check{\beta}_1^{l}]dt+\vartheta_kdW_k(t),\\
&d\check y_1^k=[Q\alpha_k-A_k^\top \check y_1^k-C^\top\check \beta_1^{k}] dt+\check \beta_1^{k}dW_k,\\
&d\check y_2^k=[-(QH+ H Q-HQH)\sum_{l=1}^K\pi_l\mathbb E\alpha_l-\sum_{l=1}^K\pi_l(F^\top \mathbb E\check y_1^l+\widetilde F^\top\mathbb E \check \beta_1^{l})-A_k^\top \check y_2^k-F^\top\sum_{l=1}^K\pi_l \check y_2^l]dt,\\
 &
\alpha_k(0)=\xi,\quad \gamma_k(T)=0,\quad \check y_1^k(T)=0,\quad \check y_2^k(T)=0,\qquad k=1,\cdots,K.
\end{aligned}\right.\end{equation}
\eqref{CC-hete-finite} is similar to the consistency condition in \cite{HHN2018} (see (2.15) therein). \cite{HHN2018} deals with mean-field game with heterogeneous case with finite diversities, hence the consistency condition only involves the Hamiltonian system of the auxiliary control problem. While for LQG-MT, besides the Hamilton system \eqref{limiting Hamilton}, CC also includes \eqref{y1,y2theta} by the person-by-person and weak-construction duality procedure.

\subsubsection{Heterogeneous case with continuum diversities but without state-coupling}
When $F=\widetilde{F}=0$, i.e., there is no weakly-coupling in state, by \eqref{minor-j} we have $\delta x_j\equiv 0$ for $j\neq i$, thus $x_j^*, x_\theta^{**}$ both vanish in \eqref{limit process of variation}. The resulting cost variation \eqref{variation-conti-1} takes a rather simple form than \eqref{cost variation-social cost-2},
\begin{equation}\begin{aligned}\label{cost variation-social cost-2-degenerate}
\delta \mathcal J_{soc}^{(N)}
 =\mathbb E\int_0^T\Big[\langle Q\bar x_i,\delta x_i\rangle-\langle (QH+ H Q-HQH)\hat x,\delta x_i\rangle
+\langle R\bar u_i,\delta u_i\rangle \Big]dt+\varepsilon_1,
\end{aligned}\end{equation}
where
\begin{equation*}\begin{aligned}
&\varepsilon_1=E\int_0^T\langle (QH+HQ- H Q H)(\hat x-\bar x^{(N)}),N\delta x^{(N)}\rangle dt.
\end{aligned}\end{equation*}
From \eqref{cost variation-social cost-2-degenerate} we can obtain the auxiliary control problem directly, i.e., it becomes unnecessary to introduce the limit terms \eqref{limit process of variation} and adjoint processes \eqref{adjoint processes}.
This is similar to the case in Section IV.A of \cite{HCM2012}. Note that in \cite{HCM2012}, there is no point-wise constraint or partial information constraint on the admissible control, hence the main focus is to find the optimal closed-loop control for the auxiliary control problem (see (32) therein). While with the above two constraints, we will obtain the  optimal open-loop control for the auxiliary control problem (see \eqref{optimal control}). In this case, \eqref{CC} reduces to
\begin{equation}\label{CC-continuum diversities}\left\{\begin{aligned}
&d\alpha=[A_\Theta\alpha+B\mathbf P_{\Gamma}[R^{-1}\mathcal E_t[-B^\top \gamma-D_{\Theta}^\top \vartheta]]]dt\\
&\qquad+[C\alpha+D_\Theta\mathbf P_{\Gamma}[R^{-1}\mathcal E_t[-B^\top \gamma-D_{\Theta}^\top \vartheta]]]dW,\\
&d\gamma=[-Q\alpha+(QH+ H Q-HQH)\mathbb E\alpha-A_\Theta^\top \gamma-C^\top \vartheta]dt+\vartheta dW(t),\\
&\alpha(0)=\xi,\quad \gamma(T)=0,
\end{aligned}\right.\end{equation}for which the well-posedness is much more easily to establish. Furthermore, if $C=D_\Theta=0$, $\Gamma=\mathbb{R}^{m}$ and $\mathbb{G}^i=\mathbb{F}^i$, by taking expectation to \eqref{CC-continuum diversities}, the derived FBSDEs reduces to the case on pp. 1740 of \cite{HCM2012}.

By contrast, when $F, \widetilde{F} \neq 0,$ variation functional $\delta \mathcal J_{soc}^{(N)}(\delta u_i)$ of \eqref{variation-conti-1} becomes rather involved depending on $x_j^*$ and $x^{**}$ both. Those two terms are some \emph{intermediate variation limits} related to basic variation term $\delta x_i$ in an indirect manner. Thus, the current representation \eqref{variation-conti-1} cannot lead a direct construction to an auxiliary control. Some duality method are required to remove dependence on these intermediate
variations.

\subsubsection{Other cases}

For homogeneous case, \cite{HHL2017} studies linear-quadratic mean-field games with control process constrained in a closed convex subset of full space $\mathbb R^m$; \cite{HWW2016} studies backward mean-filed linear-quadratic games with partial information. When there involves only constraints on the control or only  partial information, our framework is the extension of \cite{HHL2017} and \cite{HWW2016} for social optima case.

\subsection{Homogeneity and heterogeneity: a unified quasi-exchangeable approach}\label{unified quasi-exchangeable}

Recall that the mean-field theory has been extensively applied to study the large-scale weakly-coupled system along both (competitive) game and (cooperative) team directions, see e.g., \cite{BSYS2016,CD2013,HHN2018,Huang2010,HCM2007,LL2007} for recent relevant studies for game; and \cite{QHX,WZZ} for team. Essentially, such mean-field analysis is build on some \emph{exchangeability} among all individual weakly-coupled agents. It can be proved that any exchangeable sequences should be conditional independent with respect to some tail-sigma algebra. Thus, applying de Finetti theorem, the original complex weakly-coupling structure can be replaced by a deterministic- or common-noise-driven process as agent number $N$ tends to infinity. By this, all agents thus become asymptotically decoupled along with chaos propagation. Subsequently, original game or team can be reduced to low dimensional single agent optimization problem with some off-line quantities via consistency condition that matches the above exchangeable reasoning.
In this sense, mean-field analysis connects closely to exchangeable game/team in random context, and further to \emph{symmetric} game/team (\cite{CRVW}) in deterministic context. We remark that all agents in symmetric game are endowed with same underlying parameters and so become identical in analysis. So, the primal high-dimensional computation can be greatly reduced using ``mirror" argument among all symmetric agents.

Regarding large-scale system, there exist three progressive levels of diversity relevant to aforementioned exchangeability: homogeneous, heterogenous with finite/discrete diversity, and heterogenous with continuum diversity. Among them, homogenous case is most special but tractable one because all agents are statistical identical and the designed optimal team strategies should also be exchangeable. Consequently, the resulting optimized states are thus exchangeable. We refer \cite{QHX} for recent studies in such case for team, and \cite{HHL2017} for game.

Compared with homogenous case, heterogenous case with finite/discrete diversity is more realistic. Virtually, most systems in reality demonstrate some diversities in their random behaviors. In this case, all agents, from whole system scale, are no longer identical because they are endowed with diversified  parameters. However, all agents inside a sub-system with same diversity index, are still exchangeable in small scale. Thus, we can treat the large-scale system as some mixed combination of \emph{finite} exchangeable sub-systems. The previous mean-field analysis to homogenous can be suitably modified to tackle such case, with some technical but straightforward arguments. We refer \cite{AM2016} for recent studies in such case for team in discrete time setup, and \cite{Huang2010,HHN2018} for game, where a similar \emph{partial exchangeability} is introduced. %cannot cover various important applications such as portfolio selection with relative performance.}

The heterogenous case with continuum diversity, as discussed in \cite{HCM2012,NH}, should be most realistic setup for practical large-scale system. Indeed, it is less possible that the diversity of real system, can only be limited on a finite or discrete support set. Instead, considerable statistical diversity demonstrate its support on a continuum set such as compact closed interval. On the other hand, such heterogenous case should be most difficult to be handled. One reason for the continuum heterogeneity to be analytically intractable, is that the sub-class exchangeability featured in finite heterogeneity case, will shrink to zero mass along with the continuum diversity support. For this reason, the relevant results for continuum heterogeneity seems few compared with homogeneous- or finite-heterogenous-case.

We remark \cite{NH} discussed mean-field analysis with continuum diversity in game setup, and \cite{HCM2012} in team setup, using a direct state-aggregating method. However, the setting in both works are relatively simple, in particular, its weakly-coupled dynamics is only drift-controlled. This corresponds to our model with $C=D=\widetilde F=0,$ and cannot cover various applications such as portfolio selection with relative performance. Our setup is more general (diffusion-controlled and -coupled) and above aggregation method no longer works. Meanwhile, due to continuum diversity, we cannot apply the weak embedding representation method used in \cite{HHL2017,HHN2018,QHX} when tackling diffusion controlled system but of finite diversities only. Indeed, the analysis of \cite{HHN2018} replies on a construction of $K$ independent copies of optimized states with individual BMs, where $K$ is the finite cardinality of diversity. This becomes impossible for current case in presence of continuum diversities.

As resolution, this paper proposes some unified approach to homogenous-, and heterogenous-case using a quasi-exchangeable method. The main idea is as follows: first, note that the dynamics
\begin{equation*}\left\{\begin{aligned}
dx_i=&[A_{\Theta_i}x_i+Bu_i+Fx^{(N)}]dt+[Cx_i+D_{\Theta_i}u_i+\widetilde  Fx^{(N)}]dW_i,\\
x_i(0)=&\xi\in\mathbb R^n,\qquad 1\leq i\leq N,
\end{aligned}\right.\end{equation*}can be reformulated as follows:
\begin{equation*}\left\{\begin{aligned}
dx_i&=[A(z_i(t), t)x_i+Bu_i+Fx^{(N)}]dt+[Cx_i+D(z_i(t), t)u_i+\widetilde  Fx^{(N)}]dW_i,\\
dz_{i}(t)&\equiv0,\\
x_i(0)&=\xi\in\mathbb R^n; \quad
z_{i}(0)=\Theta, \qquad 1\leq i\leq N,
\end{aligned}\right.\end{equation*}
that can be further written with some augmented state as
$$d \mathbf{x}_{i}=[\mathbf{A}(\mathbf{x}_i)\mathbf{x}_i+Bu_i+\mathbf{F}\mathbf{x}^{(N)}]dt+[\mathbf{C}\mathbf{x}_i+
\mathbf{D}(\mathbf{x}_i)u_i+\widetilde{\mathbf F}\mathbf{x}^{(N)}]dW_i,\ \ \mathbf{x}_i(0)=(\xi_{i}^{\top},\ \ \Theta^{\top})^{\top}.$$
In other words, initial weakly-coupled system with continuum diversity can be viewed as some quasi-linear SDE with augmented state $\mathbf{x}_{i}=(x_{i}^{\top}, z_{i}^{\top})^{\top}$ and random initial conditions $\mathbf{x}_i(0)$ (noting $\Theta \in \mathcal{F}_{0}$, although $\xi$ might be deterministic).

To proceed, we introduce the following three systems. To ease notation, we are inclined to adopt symbols like $A(\mathbf{x})$ instead $\mathbf{A}(\mathbf{x})$ when no confusion occurs. The first system is a McKean-Vlasov SDE with random initials:
\begin{equation*}
\mathcal{P}_{1}: \quad \quad d \mathbf{x}=[{A}(\mathbf{x})\mathbf{x}+Bu+{F}\mathbb{E}\mathbf{x}]dt+[C\mathbf x+D(\mathbf x)u+\widetilde { F}\mathbb{E}\mathbf{x}]dW,\ \ \mathbf{x}(0)=(\xi^{\top},\Theta^{\top})^{\top}.
\end{equation*}
For sake of illustration, we set $\Theta \in \Lambda=\{\theta_1, \theta_2, \cdots, \theta_{K}\}$ with the mass $m_1,\cdots,m_K$ to admit finite $K$ diversity classes. Later, we will illustrate its possible extension to infinite continuum diversities. The second system is a stochastic mixture: $\widetilde{\mathbf{x}}=\sum_{j=1}^{K}m_j\widetilde{\mathbf{x}}_{j}$ but driven by identical noise $W$:
\begin{equation*}
\mathcal{P}_{2}: \quad \quad
d\widetilde{\mathbf{x}}_{j} =[A_{\theta_{j}} \widetilde{\mathbf{x}}_{j}+B u +F \mathbb{E}\widetilde{\mathbf{x}} ]dt+[C \widetilde{\mathbf{x}}_{j} +D_{\theta_{j}} u +\widetilde F \mathbb{E}\widetilde{\mathbf{x}} ]dW, \quad
\widetilde{\mathbf{x}}_{j}(0)=(\xi^{\top}, \theta_{j}^{\top})^{\top}.
\end{equation*}By contrast, the third system is also a stochastic mixture $\widehat{\mathbf{x}}=\sum_{j=1}^{K}m_j\widehat{\mathbf{x}}_{j}$ but driven by $K$ i.i.d noises $\{W_{j}\}_{j=1}^{K}$:

\begin{equation*}
\mathcal{P}_{3}: \quad \quad \begin{aligned}
d\widehat{\mathbf{x}}_{j} =[A_{\theta_{j}} \widehat{\mathbf{x}}_{j}+B u +F \mathbb{E}\widehat{\mathbf{x}} ]dt+[C \widehat{\mathbf{x}}_{j} +D_{\theta_{j}} u +\widehat F \mathbb{E}\widehat{\mathbf{x}} ]dW_{j},\quad
\widehat{\mathbf{x}}_{j}(0)=(\xi^{\top}, \theta_{j}^{\top})^{\top}.
\end{aligned}\end{equation*}
It is obvious that above three systems: $\mathbf{x}, \widetilde{\mathbf{x}}$ and $\widehat{\mathbf{x}}$ are not of the same distributions. Actually, $\mathbf{x}$ has different initial distribution at $t=0$ with $\widetilde{\mathbf{x}}, \widehat{\mathbf{x}}$, whereas $\widehat{\mathbf{x}}$ is driven by different noise with $\mathbf{x}, \widetilde{\mathbf{x}}$. Thus, all three systems are not equivalent in weak sense. However, they have same expectation dynamics, as verified using \emph{tower property} of conditional expectation, $\forall t \in [0, T]: \mathbb{E}\mathbf{x}(t)=\mathbb{E}(\mathbb{E}(\mathbf{x}(t)|\Theta))=\sum_{j=1}^{K}m_j\mathbb{E}\widetilde{\mathbf{x}}_{j}(t)
=\mathbb{E}\widetilde{\mathbf{x}}(t)=\sum_{j=1}^{K}m_j\mathbb{E}\widehat{\mathbf{x}}_{j}(t)=\mathbb{E}\widehat{\mathbf{x}}(t)$. Besides, all three systems have different second-moment function, and other finite-dimensional distributions. For example,
\begin{equation*}\begin{aligned}
&\mathbb{E}|\mathbf{x}(t)|^2=\mathbb{E}(\mathbb{E}(|\mathbf{x}(t)|^2|\Theta))=\sum_{j=1}^{K}m_j\mathbb{E}|\widetilde{\mathbf{x}}_{j}(t)|^2,\\
&\mathbb{E}|\widetilde{\mathbf x}(t)|^2=\sum_{j=1}^{K}m_j^2\mathbb{E}|\widetilde{\mathbf{x}}_{j}(t)|^2+\sum_{1\leq j<l\leq K}m_jm_l\mathbb E[\widetilde{\mathbf{x}}_{j}(t)\widetilde{\mathbf{x}}_{l}(t)],\\
&\mathbb{E}|\widehat{\mathbf x}(t)|^2=\sum_{j=1}^{K}m_j^2\mathbb{E}|\widehat{\mathbf{x}}_{j}(t)|^2+\sum_{1\leq j<l\leq K}m_jm_l\mathbb E[\widehat{\mathbf{x}}_{j}(t)\widehat{\mathbf{x}}_{l}(t)]=\sum_{j=1}^{K}m_j^2\mathbb{E}|\widetilde{\mathbf{x}}_{j}(t)|^2.
\end{aligned}\end{equation*}
Noticing above expectation equivalence is special degenerated version of Jensen inequality, thanks to the underlying LQG context. Such property cannot be extended to nonlinear moments hence $\mathbf{x}, \widetilde{\mathbf{x}}$ and $\widehat{\mathbf{x}}$ are with same expectation but different distributions.

Corresponding to $\mathcal{P}_1, \mathcal{P}_2, \mathcal{P}_3,$ we may construct three weakly-coupled systems $\mathcal{M}_1, \mathcal{M}_2, \mathcal{M}_3$:
\begin{equation*}
\mathcal{M}_{1}:  d \mathbf{x}_{i}=[{A}(\mathbf{x}_i)\mathbf{x}_i+Bu_i+{F}\mathbf{x}^{(N)}]dt+[{C}\mathbf{x}_i+
{D}(\mathbf{x}_i)u_i+\widetilde{F}\mathbf{x}^{(N)}]dW_{i},\ \ \mathbf{x}_i(0)=(\xi^{\top},\ \Theta)^{\top}.
\end{equation*}
where $\mathbf{x}^{(N)}=\frac{1}{N}\sum_{i=1}^N\mathbf{x}_{i}$. Another is weakly-coupled system $\mathcal{M}_{2}: \{\widetilde{\mathbf{x}}_{i}\}_{i=1}^{N}$ with $\widetilde{\mathbf{x}}_{i}=\sum_{j=1}^{K}m_j\widetilde{\mathbf{x}}_{i,j}$,
\begin{equation*}
\mathcal{M}_{2}:
 d\widetilde{\mathbf{x}}_{i,j} =[A_{\theta_{j}}\widetilde{\mathbf{x}}_{i,j} +B u_i +F \widetilde{\mathbf{x}}^{(N)}]dt+[C \widetilde{\mathbf{x}}_{i,j} +D_{\theta_j} u_i +\widetilde  F \widetilde{\mathbf{x}}^{(N)}]dW_i,\ \ \  \widetilde{\mathbf{x}}_{i,j}(0)=(\xi^{\top}, \theta_{j}^{\top})^{\top},
\end{equation*}
where $\widetilde{\mathbf{x}}^{(N)}=\frac{1}{N}\sum_{i=1}^N\widetilde{\mathbf{x}}_{i}$. For $1 \leq j \leq K,$ we can introduce $\widehat{\mathcal{M}}_{2}^{j}: \{\widetilde{\mathbf{x}}_{i,j}\}_{i=1}^{N}$ that is a homogeneous weakly-coupled system indexed by $\theta_{j}.$ Abusing notation, we may write informally that $\mathcal{M}_{2}=\sum_{j=1}^{K}m_{j}\widehat{\mathcal{M}}_{2}^{j},$ in other words, $\mathcal{M}_{2}$ is a finite mixture of homogeneous systems $\{\widehat{\mathcal{M}}_{2}^{j}\}_{j=1}^{K}.$ Noticing for $\widehat{\mathcal{M}}_{2}^{j},$ the driving BMs become $\{W_{i}\}_{i=1}^{N}$ which are same to that of $\widehat{\mathcal{M}}_{2}^{j'}$ for $j \neq j'.$ Thus, totally there involve $N$ independent BMs for $\mathcal{M}_{2}$. Moreover, if we introduce a sampling sequence from $\{1, \cdots, K\}$ with $\mathcal{I}_{j}=\{ \theta_{i}=j, 1\leq i \leq N\}$ and $\lim_{N \rightarrow +\infty} \frac{\text{Card} \mathcal{I}_j}{N}=m_j, \ \ 1 \leq j \leq K.$ Then, $\mathcal{M}_{2}$ is equivalent in weak sense to stochastic $K-$heterogenous weakly-coupled system introduced in \cite{HHN2018,{Huang2010}}.

The third system is $\mathcal{M}_{3}: \{\widehat{\mathbf{x}}_{i}\}_{i=1}^{N}$ with $\widehat{\mathbf{x}}_{i}=\sum_{j=1}^{K}m_j\widehat{\mathbf{x}}_{i,j}$,
\begin{equation*}
\mathcal{M}_{3}:
 d\widehat{\mathbf{x}}_{i,j} =[A_{\theta_{j}}\widehat{\mathbf{x}}_{i,j} +B u_i +F \widehat{\mathbf{x}}^{(N)}]dt+[C \widehat{\mathbf{x}}_{i,j} +D_{\theta_j} u_i +\widehat F \widehat{\mathbf{x}}^{(N)}]dW_{i,j},\ \ \  \widehat{\mathbf{x}}_{i,j}(0)=(\xi^{\top}, \theta_{j}^{\top})^{\top},
\end{equation*}
where $\widehat{\mathbf{x}}^{(N)}=\frac{1}{N}\sum_{i=1}^N\widehat{\mathbf{x}}_{i}$. For $1 \leq j \leq K,$ we can introduce $\widehat{\mathcal{M}}_{3}^{j}: \{\widehat{\mathbf{x}}_{i,j}\}_{i=1}^{N}$ that is a homogeneous weakly-coupled system indexed by $\theta_{j}.$ Noticing for $\widehat{\mathcal{M}}_{3}^{j},$ the driving BMs become $\{W_{i,j}\}_{i=1}^{N}$. So, totally there arise $N \times K$ independent BMs for $\mathcal{M}_{3},$ or re-scale to $N$ BMs for each sub-system $\widehat{\mathcal{M}}_{3}^{j}, 1 \leq j \leq K.$ This is not problematic when $K$ is finite. Again, $\mathcal{M}_{3}$ is finite mixture of homogeneous system $\{\widehat{\mathcal{M}}_{3}^{j}\}_{j=1}^{K}.$ We remark that $\widehat{\mathcal{M}}_{3}^{j}$ and $\widehat{\mathcal{M}}_{2}^{j}$ are driven by different BMs, but they are equivalent weak-coupled homogenous system in weak sense. This is because they share have same state-average limit by law of large numbers, although they are driven by different BMs systems.

Moreover, we can introduce an augmented state ${\mathbf{y}}_{i}=(\widehat{\mathbf{x}}_{i,1}^{\top}, \cdots, \widehat{\mathbf{x}}_{i,K}^{\top})^{\top}$ and $\widehat{\mathbf x}^{(N)}=\frac{1}{N}\sum_{i=1}^N\widehat{\mathbf x}_i$, it follows that
\begin{equation*}
 d{\mathbf{y}}_{i}=[\widehat A{\mathbf{y}}_{i} +\widehat B \widehat u_i +  \mathbf F \mathbf y^{(N)}]dt+\sum_{j=1}^K[\widehat C_j {\mathbf{y}}_{i} +\widehat D_{j} u_i +\widehat{\mathbf F}_j \mathbf{y}^{(N)}]dW_{i,j},\ \ {\mathbf{y}}_{i}(0)=(\xi^{\top}, \theta_{1}^{\top} \cdots, \xi^\top, \theta_{K}^{\top})^{\top},
\end{equation*}
where
\begin{equation*}
  \begin{aligned}
     & \widehat{A}=
    \left(\begin{smallmatrix}
      A_{\theta_1}   & \cdots & 0 \\
      \vdots  & \ddots &\vdots\\
      0&  \cdots & A_{\theta_K}  \\
    \end{smallmatrix}\right)_{(nK\times nK)},
    %_{(Nn\times 1)},
    \widehat{B}=
    \left(\begin{smallmatrix}
      B &  \cdots & 0 \\
      %0 & B & \cdots & 0 \\
     \vdots  & \ddots &\vdots\\
      0 &  \cdots & B \\
    \end{smallmatrix}\right)_{(nK\times mK)},
\widehat u_i=
    \left(\begin{smallmatrix}
     u_i \\
      \vdots\\
      u_i
    \end{smallmatrix}\right)_{(mK\times 1)},\\
&   \mathbf{F}=
   \left(\begin{smallmatrix}
     Fm_1&\cdots&Fm_K \\
     \vdots &\vdots&\vdots\\
      Fm_1&\cdots&Fm_K \\
    \end{smallmatrix}\right)_{(nK\times nK)},\widehat{C}_j=\begin{smallmatrix}
      1 \\
      \vdots \\
      j \\
      \vdots \\
      K
    \end{smallmatrix}
    \left(\begin{smallmatrix}
      0 & \cdots & 0 & \cdots & 0\\
      \vdots & \vdots &  \vdots & \vdots & \vdots\\
      0 & \cdots &  C
      &\cdots & 0\\
      \vdots & \vdots & \vdots & \vdots & \vdots\\
      0 & \cdots &  0 & \cdots & 0\\
    \end{smallmatrix}\right)_{(nK\times nK)},\\
     &
     \widehat{D}_j=\begin{smallmatrix}
      1 \\
      \vdots \\
      j \\
      \vdots \\
      K
    \end{smallmatrix}
    \left(\begin{smallmatrix}
      0 & \cdots &  0 & \cdots & 0\\
      \vdots & \vdots &  \vdots & \vdots & \vdots\\
      0 & \cdots &  D_{\theta_j}
      &\cdots & 0\\
      \vdots & \vdots  & \vdots & \vdots & \vdots\\
      0 & \cdots &  0 & \cdots & 0\\
    \end{smallmatrix}\right)_{(nK\times mK)},
    \widehat{\mathbf F}_j=\begin{smallmatrix}
      1 \\
      \vdots \\
      j \\
      \vdots \\
      K
    \end{smallmatrix}
    \left(\begin{smallmatrix}
      0&\cdots&0 \\
      \vdots&\vdots&\vdots\\
     \widehat Fm_1&\cdots&\widehat Fm_K\\
      \vdots &\vdots&\vdots\\
      0&\cdots&0 \\
    \end{smallmatrix}\right)_{(nK\times nK)}.
    %,
    %\Xi = \left(\begin{smallmatrix}
%      \xi\\
%      \vdots\\
%      \xi
%    \end{smallmatrix}\right).
    %_{(Nn\times 1)}.
  \end{aligned}
\end{equation*}
It follows that $\mathcal{M}_{3}: \{\widehat{\mathbf{x}}_{i}\}_{i=1}^{N}$ satisfying $\widehat{\mathbf{x}}_{i}=\textbf{m} \cdot {\mathbf{y}}_{i}$ with $\textbf{m}=(m_1, \cdots, m_{K}).$ Noticing that $\{{\mathbf{y}}_{i}\}_{i=1}^{N}$ is homogenous for $1 \leq i \leq N$ and so is the case for $\{\widehat{\mathbf{x}}_{i}\}_{i=1}^{N}$, thus $\mathcal{M}_{3}$ can be viewed as a homogenous system but with augmented state ${\mathbf{y}}_{i}.$
Hence, $\mathcal{M}_{3}$ can be formulated either a finite mixture of $K$-homogeneous system $\{\widehat{\mathcal{M}}_{3}^{j}\}_{j=1}^{K},$ or a single homogenous system but with augmented (mixed) state ${\mathbf{y}}_{i}.$ Note that the later formulation on augmented ${\mathbf{y}}_{i}$ actually connects to the so-called direct method (\cite{WZZ}). In fact, by formulation on ${\mathbf{y}}_{i},$ we can apply direct method proposed in (\cite{WZZ}) for homogenous system only but now on more intractable (finite) heterogenous system. As the trade-off, the associated Riccati or Hamiltonian system become augmented accordingly with coupled block structure due to $K$ diversity.

The above three weakly-coupled systems $\mathcal{M}_1, \mathcal{M}_2, \mathcal{M}_3$ have different distributions but always with the same asymptotic empirical state-average as $N\rightarrow+\infty$. In fact, they are generated from same underlying weakly-coupled stochastic systems but differs in filtration on given timing point. To be precise, all agents in $\mathcal{M}_1$ are exchangeable in quasi-sense (at filtration point $\mathcal{F}_{0}$) before the diversity sampling. In this case, $\mathbf{x}_{i}(t)=\mathbb{E}(\mathbf{x}_{i}(t)|\mathcal{F}_{t})=\mathbb{E}(\mathbf{x}_{i}(t)|\Theta, W_{i}(s), 0 \leq s \leq t, 1 \leq i \leq N).$ On the other hand, $\mathcal{M}_2$ is the same system but conditional on the pre-sampled diversity index $\Theta_i.$ In this case, $\widetilde{\mathbf{x}}_{i}(t)=\mathbb{E}(\mathbb{E}(\mathbf{x}_{i}(t)|\Theta)|W_{i}(s), 0 \leq s \leq t, 1 \leq i \leq N).$ Last, $\mathcal{M}_3$ is same weak-coupled system but after the sampling of diversity $\Theta$ and $\widehat{\mathcal{M}}_{3}^{j}$ is just the re-labeled system with realization $\Theta=\theta_j.$ In this sense, all three systems $\mathcal{M}_{1}, \mathcal{M}_{2}, \mathcal{M}_{3}$ characterize the same underlying dynamics but from different temporal section. Thus, they are equivalent for mean-field analysis because they share the same state-average limit (in formulation, and Step 1 for decomposition) and expectation operator (in Step 3 for CC).

To recap, we present the following diagram where ``$\Longleftrightarrow$" represents the equivalent expectation operator in first line, while asymptotic state-average operator in second line:
\begin{equation}\left\{\label{equivalent ralation}\begin{aligned}
&\text{single-agent:} \ \mathcal{P}_1 \Longleftrightarrow \mathcal{P}_2 \Longleftrightarrow \mathcal{P}_3,\\
&\text{weakly-coupled agents:}\  {\mathcal{M}_1} \Longleftrightarrow \mathcal{M}_2 \Longleftrightarrow \mathcal{M}_3 \Longleftrightarrow \mathcal{M}\ \text{(stochastic $K$-heterogenous system),} \\& \mathcal{M}_1: \text{homogenous but with random diversity index $\Theta$, augmented randomness, pre-sampling}\\& \mathcal{M}_2: \text{mixture of $K$ homogenous system, pre-sampling}\\&\mathcal{M}_3: \text{homogenous system with (augmented) mixture of states, post-sampling}
\\&\mathcal{M}: \text{$K$ heterogenous system defined by relative frequency of diversity sequence, post-sampling.}
\end{aligned}\right.\end{equation}
Above arguments in \eqref{equivalent ralation} are on the basis that $\Theta$ is finite-valued only. Now we present its generalization to case when $\Theta$ has continuum diversity support. In this case, we have
\begin{equation*}\left\{\begin{aligned}
& \mathcal{M}_1^c: \text{homogenous but with random diversity index $\Theta$, augmented randomness, pre-sampling}\\
& \mathcal{M}_2^c: \text{mixture of continuum homogenous system, pre-sampling}\\
&\mathcal{M}_3^c: \text{homogenous system with (augmented) mixture of states, post-sampling}
\\&\mathcal{M}^c: \text{continuum heterogenous system defined by empirical distribution of diversity}\\
 &\text{\qquad sequence, post-sampling.}
\end{aligned}\right.\end{equation*}
$\mathcal{M}_1^c$ is still well-defined and we have already proceeded the analysis as in Section \ref{auxiliary problem}. On the other hand, $\mathcal{M}_{3}^c$ is no longer well defined since now we have to introduce continuum-valued BMs for $\widehat{\mathcal{M}}_{3}^{\theta,c}$ to model the diversity. By contrast, $\mathcal{M}_2^c$ is still well defined since we need still only formulate countable BMs for each $\widehat{\mathcal{M}}_{2}^{\theta,c}, \theta \in \mathcal{S}$, and in total, only countable BMs are still invoked. In this case, we may further set $\widetilde{\mathbf{x}}_{i}=\int_{\mathcal{S}}\widetilde{\mathbf{x}}_{i, \theta}d\Phi(\theta)$ and proceed the classical mean-field analysis as in \cite{HCM2012}. However, classical mean-field analysis only works on $\mathcal{M}_2^c$ with $C=D=F=\widetilde F=0$. In general case with $F, \widetilde F \neq 0,$ such classical analysis fails because its CC system should invoke an embedding representation (see e.g., \cite{HWY2019}), and a continuum-valued BMs system will be required to replicate the distribution for a generic agent who is still continuum-heterogenous (diversified). Moreover, in \cite{HCM2007}, the continuum heterogeneity is defined through some limiting empirical distribution by Glivenko-Cantelli Lemma. Note that the continuum set therein is required to be compact when using Glivenko-Cantelli arguments, while in our framework of $\mathcal M_1^c$, such compactness is not required. Consequently, this paper can deal with general continuum diversity based on formulation $\mathcal{M}_1^c$, as summarized as follows.

First, we can verify that $\mathcal{M}_1^c, \mathcal{M}_2^c$ as well as $\mathcal{M}^c$ (note that $\mathcal{M}_3^c$ becomes infeasible to be defined) are still of the same asymptotic state-average limit. In this sense, the generic agents in $\mathcal{M}^c$ are {quasi-exchangeable} because although they are not exchangeable after diversity sapling, but $\mathcal{M}^c$ shares the same expectation and asymptotic state-average limit with $\mathcal{M}_1^c, \mathcal{M}_2^c$, and all agents of $\mathcal{M}_1^c$ are exchangeable before the sampling. Second, given such quasi-exchangeable property, the original $\mathcal{M}^c$ or $\mathcal{M}_{2}^c$ system with continuum heterogeneity can be converted to $\mathcal{M}_1^c$ that is a homogenous one but with augmented randomness ($\{\Theta_i, W_{i}\}_{i=1}^{N})$ as trade-off. Third, as discussed in Section \ref{auxiliary problem}, some new type variation-decomposition and auxiliary control problem can thus be constructed, and CC condition can be represented via some weak-construction on continuum diversity support as in Theorem \ref{CC of LQGMT}.

\section{Wellposedness of consistency condition}\label{well-posedness of CC}

This section continues to complete (Step 3) by establishing some well-posedness to consistency condition derived in Section \ref{Consistency condition}. Note that \eqref{CC} is fully-couple FBSDEs involved with double projections whose well-posedness cannot be guaranteed by current literature. Moreover, as explained in Section \ref{unified quasi-exchangeable},  \eqref{CC} is obtained by converting system with continuum heterogeneity to a homogenous one but with augmented randomness ($\{\Theta_i, W_{i}\}_{i=1}^{N})$ as trade-off. Based on this, we will apply the discounting method to study \eqref{CC} which would provide some mild conditions to ensure the existence and uniqueness of fully-coupled FBSDEs as \eqref{CC}.
Define $X=\alpha$, $Y=(\gamma^\top,\check{y}_1^\top,(\check y_2^\theta)^\top)^\top$ and $Z=(\vartheta^\top,\check\beta_1^\top,0)^\top$. For simplicity, let $\mathcal E_t[Y]=\mathbb E[Y|\mathcal G_t]$ and $\mathcal E_t[Z]=\mathbb E[Z|\mathcal G_t]$, $\widetilde{\mathbb E}[Y]=((\int_\mathcal{S}\gamma d\Phi(\theta))^\top,(\int_\mathcal{S}\check y_1 d\Phi(\theta))^\top,(\int_\mathcal{S}\check y_2^\theta d\Phi(\theta))^\top)^\top$,
%$$Z=\left(
%      \begin{array}{ccccccccc}
%        \vartheta_1 & \cdots & 0 & \check\beta_1^1 & \cdots & 0 & 0 & \cdots & 0\\
%        \vdots & \ddots & \vdots & \vdots & \ddots & \vdots&\vdots & \ddots & \vdots\\
%        0 & \cdots & \vartheta_K & 0 & \cdots & \check\beta_1^K & 0 & \cdots & 0\\
%      \end{array}
%    \right)^\top,$$
    then \eqref{CC} takes the following form:
\begin{equation}\label{CC-equivalent form}\left\{\begin{aligned}
&dX=[A_\Theta X+F\mathbb E[X]+\mathbb B_1(Y,Z)]dt+[CX+\widetilde F\mathbb E[X]+\mathbb D_\Theta(Y,Z)]dW,\\
&dY=[\mathbb A_2X+\bar{\mathbb A}_2\mathbb E[X]+\mathbb B_2Y+\bar{\mathbb B}_2\mathbb E[Y]+\widetilde{\mathbb B}_2\widetilde{\mathbb E}[Y]+\mathbb C_2Z+\bar{\mathbb C}_2\mathbb E[Z]]dt+ZdW,\\
&X(0)=\xi,\qquad Y(T)=(0,\cdots,0)^\top,
\end{aligned}\right.\end{equation}
where
\begin{equation*}\left\{\begin{aligned}
&\mathbb B_1(Y,Z)=B\mathbf P_{\Gamma}[R^{-1}\mathcal E_t[-B^\top \gamma-D_{\Theta}^\top \vartheta]]\\
&=
     B\mathbf P_{\Gamma}[R^{-1}((-B^\top,0,\cdots,0)\mathcal E_t[Y]+(-D_\Theta^\top,0,\cdots,0)\mathcal E_t[Z])],\\
&\mathbb D_\Theta(Y,Z)=D_\Theta\mathbf P_{\Gamma}[R^{-1}\mathcal E_t[-B^\top \gamma-D_{\Theta}^\top \vartheta]]\\
&=
     D_\Theta\mathbf P_{\Gamma}[R^{-1}((-B^\top,0,\cdots,0)\mathcal E_t[Y]+(-D_\Theta^\top,0,\cdots,0)\mathcal E_t[Z])],\\
 &%\mathbb A_2=\left(
%              \begin{smallmatrix}
%              Q& Q& 0 &\cdots & 0\\
%              0&0 & Q & \cdots & 0  \\
%              \vdots & \vdots & \vdots & \ddots & \vdots \\
%              0  &0 & 0 & \cdots & Q \\ \end{smallmatrix}
%            \right),
\mathbb A_2=\left(
              \begin{smallmatrix}
              -Q\\
               Q  \\
               0\\
               \end{smallmatrix}
            \right),
            \bar{\mathbb A}_2=\left(
                    \begin{smallmatrix}
                    QH+HQ-HQH\\
                    0 \\
                     -(QH+HQ-HQH)\\
                    \end{smallmatrix}
                  \right),\mathbb B_2=\left(
               \begin{smallmatrix}
-A_\Theta^\top & 0 & 0\\
0 & -A_\Theta^\top & 0 \\
0 & 0 & -A_\theta^\top -F^\top\\
               \end{smallmatrix}
             \right),
\bar{\mathbb B}_2=\left(
               \begin{smallmatrix}
0 & F^\top  & 0 \\
0  & 0 & 0 \\
0 & -F^\top  &0 \\
               \end{smallmatrix}
             \right),\\
                  &\widetilde{\mathbb B}_2=\left(
               \begin{smallmatrix}
0 & 0 & F^\top\\
0 & 0 & 0 \\
0 & 0 &0\\
               \end{smallmatrix}
             \right),
\mathbb C_2=\left(
               \begin{smallmatrix}
-C^\top & 0 & 0\\
0 & -C^\top & 0 \\
0 & 0 & 0\\
               \end{smallmatrix}
             \right),\bar{\mathbb C}_2=\left(
               \begin{smallmatrix}
0 & \widetilde F^\top & 0\\
0 & 0 & 0 \\
0 & -\widetilde F^\top & 0\\
               \end{smallmatrix}
             \right),
\end{aligned}\right.\end{equation*}
and $0$ denotes the zero vector or zero matrix with suitable dimensions. Note that in \eqref{CC-equivalent form}, $\widetilde{\mathbb B}_2\widetilde{\mathbb E}[Y]=\widetilde{\mathbb E}[\widetilde{\mathbb B}_2 Y]=\mathbb E[\widetilde{\mathbb B}_2 Y].$
 To start, we first give some results for general nonlinear mean-field forward-backward system with double projections:
\begin{equation}\label{nonlinear FBSDE}\left\{\begin{aligned}
&dX=b(t,X,\mathbb E[X],Y,\mathcal E_t[Y],Z,\mathcal E_t[Z])dt+\sigma(t,X,\mathbb E[X],Y,\mathcal E_t[Y],Z,\mathcal E_t[Z])dW,\\
&dY(t)=-f(t,X,\mathbb E[X],Y,\mathbb E[Y],\widetilde{\mathbb E}[Y],Z,\mathbb E[Z])dt+ZdW,\\
&X(0)=x,\qquad Y(T)=0,
\end{aligned}\right.\end{equation}
where $\mathbb E[\widetilde{\mathbb E}[Y]]=\mathbb E[Y]$ and the coefficients satisfy the following conditions:
\begin{description}
  \item[(H1)] There exist $\rho_1,\rho_2\in\mathbb R$ and positive constants $k_i,i=1,\cdots,17$ such that
  for all $t\in[0,T]$, $x,x_1,x_2,\bar x,\bar x_1,\bar x_2\in\mathbb R^n$, $y,y_1,y_2,\bar y,\bar y_1,\bar y_2,\hat y,\hat y_1,\hat y_2,\widetilde y,\widetilde y_1,\widetilde y_2\in\mathbb R^m$, $z,z_1,z_2,\bar z,\bar z_1,\bar z_2,\widetilde z,\\\widetilde z_1,\widetilde z_2\in\mathbb R^{m}$, a.s.,
\begin{equation*}\begin{aligned}
&\langle b(t,x_1,\bar x,y,\hat y,z,\hat z)-b(t,x_2,\bar x,y,\hat y,z,
\hat z),x_1-x_2\rangle\leq \rho_1|x_1-x_2|^2,\\
&|b(t,x,\bar x_1,y_1,\hat y_1,z_1,\hat z_1)-b(t,x,\bar x_2,y_2,\hat y_2,z_1,\hat z_2)|\\
\leq& k_1|\bar x_1-\bar x_2|+k_2|y_1-y_2|+k_3|\hat y_1-\hat y_2|+k_4|z_1-z_2|+k_5|\hat z_1-\hat z_2|,\\
&\langle f(t,x,\bar x,y_1,\bar y,\widetilde y,z,\bar z)-f(t,x,\bar x,y_2,\bar y,\widetilde y,z,\bar z),y_1-y_2\rangle\leq\rho_2|y_1-y_2|^2,\\
&|f(t,x_1,\bar x_1,y,\bar y_1,\widetilde y_1,z_1,\bar z_1)-f(t,x_2,\bar x_2,y,\bar y_2,\widetilde y_2,z_2,\bar z_2)|\\
\leq& k_6|x_1-x_2|+k_7|\bar x_1-\bar x_2|+k_8|\bar y_1-\bar y_2|+k_9|\widetilde y_1-\widetilde y_2|+k_{10}|z_1-z_2|+k_{11}|\bar z_1-\bar z_2|,\\
&|\sigma(t,x_1,\bar x_1,y_1,\hat y_1,z_1,\hat z_1)-\sigma(t,x_2,\bar x_2,y_2,\hat y_2,z_2,\hat z_2)|^2\\
\leq& k_{12}^2|x_1-x_2|^2+k_{13}^2|\bar x_1-\bar x_2|^2+k_{14}^2|y_1-y_2|^2+k_{15}^2|\hat y_1-\hat y_2|^2+k_{16}^2|z_1-z_2|^2+k_{17}^2|\hat z_1-\hat z_2|^2.
\end{aligned}\end{equation*}
\item[(H2)]
\begin{equation*}
\mathbb E\int_0^T\Big[|b(t,0,0,0,0,0,0)|^2+|\sigma(t,0,0,0,0,0,0)|^2+|f(t,0,0,0,0,0,0,0)|^2\Big]dt<\infty.
\end{equation*}
\end{description}
Similar to \cite{HHN2018} and \cite{PT1999}, we have the following result of the solvability of \eqref{nonlinear FBSDE}. For the readers' convenience, we give the proof in the appendix.
\begin{theorem}\label{discounting}
Suppose \emph{(H1)} and \emph{(H2)} hold. There exists a constant $\delta_1>0$ depending on $\rho_1,\rho_2,T,k_i,i=1,8,9,10,11,13$ such that if $k_i\in[0,\delta_1)$, $i=2,3,4,5,6,7,12,14,15,16,17$, FBSDEs \eqref{nonlinear FBSDE} admits a unique adapted solution $(X,Y,Z)\in L^2_{\mathcal F}(0,T;\mathbb R^n)\times L^2_{\mathcal F}(0,T;\mathbb R^m)\times L^2_{\mathcal F}(0,T;\mathbb R^{m}).$ Furthermore, if $2\rho_1+2\rho_2<-2k_1-2k_8-2k_9-k_{10}^2-k_{11}^2-k_{12}^2-k_{13}^2$,
there exists a constant $\delta_2>0$ depending on $\rho_1,\rho_2,k_i,i=1,8,9,10,11,13$ such that if $k_i\in[0,\delta_2)$, $i=2,3,4,5,6,7,12,14,15,16,17$, FBSDEs \eqref{nonlinear FBSDE} admits a unique adapted solution $(X,Y,Z)\in L^2_{\mathcal F}(0,T;\mathbb R^n)\times L^2_{\mathcal F}(0,T;\mathbb R^m)\times L^2_{\mathcal F}(0,T;\mathbb R^{m}).$
\end{theorem}
%Note that the projection operator is Lipschitz continuous with Lipschitz constant  1.
Let $\rho_1^*=\mathop{esssup}\limits_{0\leq s\leq T}\mathop{esssup}\limits_{\theta\in\mathcal S}\Lambda_{\max}(-\frac{1}{2}(A_\theta(s)+A_\theta(s)^\top))$ and $\rho_2^*=\mathop{esssup}\limits_{0\leq s\leq T}\Lambda_{\max}(-\frac{1}{2}(\mathbb B(s)+\mathbb B(s)^\top))$, where $\Lambda_{\max}(M)$ is the largest eigenvalue of the matrix $M$. For $M(\cdot)\in L^\infty_{\mathbb F}(0,T;\mathbb R^{n\times n})$, $\|M(\cdot)\|\triangleq\mathop{esssup}\limits_{0\leq s\leq T}\mathop{esssup}\limits_{\omega\in\Omega}\|M(s)\|$.
Comparing \eqref{nonlinear FBSDE} with \eqref{CC-equivalent form}, we can check that the parameters of (H1) and (H2) can be chosen as follows:
\begin{equation*}\begin{aligned}
&k_1=\|F\|,k_2=k_4=k_{12}=k_{14}=0, k_3=\| B\|\|R^{-1}\|\| B\|, k_5=\| B\|\|R^{-1}\|\| D_{\Theta}\|,  \\
& k_6=\|\mathbb A_2\|,k_7=\|\bar{\mathbb A}_2\|, k_8=\|\bar{\mathbb B}_2\|, k_9=\|\widetilde{\mathbb B}_2\|, k_{10}=\|\mathbb C_2\|, k_{11}=\|\bar{\mathbb C}_2\|, k_{12}=\sqrt{3}\|C\|,\\ & k_{13}=\sqrt{3}\|\widetilde F\|,k_{15}=\sqrt{6}\|{D}_{\Theta}\|\| R^{-1}\|\| B\|, k_{17}=\sqrt{6}\|{D}_{\Theta}\|\| R^{-1}\|\| D_{\Theta}\|.
\end{aligned}\end{equation*}
Now we introduce the following assumption:
\begin{description}
  \item[(A4)] $2\rho_1^*+2\rho_2^*<-2\|F\|-2\|\bar{\mathbb B}_2\|-2\|\widetilde{\mathbb B}_2\|-\|\mathbb C_2\|^2-\|\bar{\mathbb C}_2\|^2-3\|C\|^2-3\|\widetilde F\|^2.$
\end{description}
It follows from Theorem \ref{discounting} that
\begin{proposition}
Under \emph{(A4)},
there exists a constant $\delta_3>0$ depending on $\rho_1^*,\rho_2^*,k_i,i=1,8,9,10,11,13$ such that if $k_i\in[0,\delta_3)$, $i=3,5,6,7,15,16,17$, FBSDEs \eqref{CC-equivalent form} admits a unique adapted solution $(X,Y,Z)\in L^2_{\mathcal F}(0,T;\mathbb R^n)\times L^2_{\mathcal F}(0,T;\mathbb R^{3n})\times L^2_{\mathcal F}(0,T;\mathbb R^{3n}).$
\end{proposition}

\section{Asymptotic $\varepsilon$-optimality}\label{asymptotic optimality}
This section aims to complete (Step 4) so as to verify the asymptotic optimality of mean-field team strategy derived in Section \ref{decentralized strategy}.
Here we proceed our verification based on the assumption in Section \ref{well-posedness of CC}, i.e., (A4).

\subsection{Representation of social cost}\label{Representation of social cost}
First, we give a quadratic representation of the team functional. Rewrite the large-population system \eqref{state equation} as follows:
\begin{equation}\label{LP system}
  \begin{aligned}d\mathbf{x} = (\mathbf{A}\mathbf{x} + \mathbf{B}u)dt + \sum_{i = 1}^{N}(\mathbf{C}_i\mathbf{x} + \mathbf{D}_iu)dW_i, \qquad\mathbf{x}(0) = \widetilde{\xi},\\
  \end{aligned}
\end{equation}
where
\begin{equation*}
  \begin{aligned}
     & \mathbf{A}=
    \left(\begin{smallmatrix}
      A_{\Theta_1} + \frac{F}{N} & \frac{F}{N} & \cdots & \frac{F}{N} \\
      \frac{F}{N} & A_{\Theta_2} + \frac{F}{N} & \cdots & \frac{F}{N} \\
      \vdots & \vdots & \ddots &\vdots\\
      \frac{F}{N} & \frac{F}{N} & \cdots & A_{\Theta_N} + \frac{F}{N} \\
    \end{smallmatrix}\right),
    %_{(Nn\times Nn)},
    \mathbf{x}=
    \left(\begin{smallmatrix}
      x_1 \\
      \vdots\\
      x_N
    \end{smallmatrix}\right),
    %_{(Nn\times 1)},
    \mathbf{B}=
    \left(\begin{smallmatrix}
      B & 0 & \cdots & 0 \\
      0 & B & \cdots & 0 \\
     \vdots & \vdots & \ddots &\vdots\\
      0 & 0 & \cdots & B \\
    \end{smallmatrix}\right),
    %_{(Nn\times Nm)},
u=
    \left(\begin{smallmatrix}
      u_1 \\
      \vdots\\
      u_N
    \end{smallmatrix}\right),\\
    %_{(Nm\times 1)},\\
    & \mathbf{C}_i=\begin{smallmatrix}
      1 \\
      \vdots \\
      i \\
      \vdots \\
      N
    \end{smallmatrix}
    \left(\begin{smallmatrix}
      0 & \cdots & 0 & 0 & 0 & \cdots & 0\\
      \vdots & \vdots & \vdots & \vdots & \vdots & \vdots & \vdots\\
      \frac{\widetilde{F}}{N} & \cdots &\frac{\widetilde{F}}{N} & \frac{\widetilde{F}}{N} + C &\frac{\widetilde{F}}{N}
      &\cdots & \frac{\widetilde{F}}{N}\\
      \vdots & \vdots & \vdots & \vdots & \vdots & \vdots & \vdots\\
      0 & \cdots & 0 & 0 & 0 & \cdots & 0\\
    \end{smallmatrix}\right),
    %_{(Nn\times Nn)},
    \mathbf{D}_i=\begin{smallmatrix}
      1 \\
      \vdots \\
      i \\
      \vdots \\
      N
    \end{smallmatrix}\left(\begin{smallmatrix}
      0 & \cdots & 0 & 0 & 0 &  \cdots & 0\\
      \vdots & \vdots & \vdots & \vdots & \vdots & \vdots & \vdots\\
      0 & \cdots & 0  & D_{\Theta_i} &0  & \cdots  & 0\\
      \vdots & \vdots & \vdots & \vdots & \vdots & \vdots & \vdots\\
      0 & \cdots & 0  & 0 &0  & \cdots  & 0\\
    \end{smallmatrix}\right),
    %_{(Nn\times Nm)},
    \widetilde\xi = \left(\begin{smallmatrix}
      \xi\\
      \vdots\\
      \xi
    \end{smallmatrix}\right).
    %_{(Nn\times 1)}.
  \end{aligned}
\end{equation*}
Similarly, the social cost takes the following form:
\begin{equation*}
  \begin{aligned}
  \mathcal{J}^{(N)}_{soc}(u)
   =& \frac{1}{2}\sum_{i=1}^n\mathbb E\int_0^T\Big[\langle Q(x_i- H x^{(N)}),(x_i- Hx^{(N)})\rangle+\langle Ru_i,u_i\rangle\Big]dt                    \\
   =& \frac{1}{2}\mathbb{E}\int_{0}^{T}\Big[ \langle \mathbf{Q}\mathbf{x},\mathbf{x}\rangle +\langle \mathbf{R}u,u\rangle \Big]dt, \\
  \end{aligned}
\end{equation*}
where
{\small{\begin{equation*}
  \begin{aligned}
     & \mathbf{Q}=
    \left(\begin{smallmatrix}
 Q + \frac{1}{N}(H^\top QH - QH - H^\top Q) &   \frac{1}{N}(H^\top QH - QH - H^\top Q) &  \cdots & \frac{1}{N}(H^\top QH - QH - H^\top Q) \\
  \frac{1}{N}(H^\top QH - QH - H^\top Q) & Q + \frac{1}{N}(H^\top QH - QH - H^\top Q) & \cdots & \frac{1}{N}(H^\top QH - QH - H^\top Q) \\
   \vdots & \vdots & \ddots &\vdots\\
 \frac{1}{N}(H^\top QH - QH - H^\top Q) & \frac{1}{N}(H^\top QH - QH - H^\top Q) & \cdots & Q + \frac{1}{N}(H^\top QH - QH - H^\top Q) \\
    \end{smallmatrix}\right),
    %_{(Nn\times Nn)}  ,                                                                                      \\
  \mathbf{R} = \left(\begin{smallmatrix}
      R & 0 & \cdots & 0\\
      0 & R & \cdots & 0\\
      \vdots & \vdots & \ddots & \vdots\\
      0 & 0 & \cdots & R\\
    \end{smallmatrix}\right).
    %_{(Nm\times Nm)}.
  \end{aligned}
\end{equation*}}}
Next, by the variation of constant formula, we know that the strong solution of \eqref{LP system} admits the following representation:
\begin{equation*}
  \begin{aligned}
    \mathbf{x}(t) =  \Phi(t)\widetilde\xi + \Phi(t)\int_{0}^{t} \Phi(s)^{-1}[(\mathbf{B} - \sum_{i = 1}^{N}\mathbf{C}_i\mathbf{D}_i)u(s)]ds + \sum_{i = 1}^{N}\Phi(t) \int_{0}^{t}\Phi(s)^{-1}\mathbf{D}_iu(s)dW_i(s),
  \end{aligned}
\end{equation*}
where
\begin{equation*}
  \begin{aligned}
      d\Phi(t) = \mathbf{A}\Phi(t)dt + \sum_{i = 1}^{N}\mathbf{C}_i\Phi(t)dW_i(t), \qquad \Phi(0) = I.
  \end{aligned}
\end{equation*}
Define the following operators
\begin{equation*}
  \left\{
  \begin{aligned}
     &\phi( u)(\cdot) := \Phi(\cdot)\Big\{\int_{0}^{\cdot}\Phi(s)^{-1}[(\mathbf{B} - \sum_{i = 1}^{N}\mathbf{C}_i\mathbf{D}_i)u(s)]ds+ \sum_{i = 1}^{N} \int_{0}^{\cdot}\Phi(s)^{-1}\mathbf{D}_iudW_i(s)\Big\} \\
     & \widetilde{\phi}(u) :=\phi (u)(T),\quad \mathcal{S} (y)(\cdot) := \Phi(\cdot)\Phi^{-1}(0)\widetilde\xi,\quad \widetilde{\mathcal{S}}(y) :=\mathcal{S} (y)(T),
  \end{aligned}
  \right.
\end{equation*}
then for any admissible control $u$, we have
\begin{equation*}
  \begin{aligned}
      \mathbf{x}(\cdot) =\phi (u)(\cdot) + \mathcal{S} (y)(\cdot), \qquad \mathbf{x}(T) = \widetilde{\phi}(u) + \widetilde{\mathcal{S}}(y).
  \end{aligned}
\end{equation*}
Note that $\phi(\cdot):(L^2_{\mathcal F}(0,T;\Gamma),\cdots, L^2_{\mathcal F}(0,T;\Gamma))\rightarrow (L^2_{\mathcal F}(0,T;\mathbb R^n),\cdots,L^2_{\mathcal F}(0,T;\mathbb R^n))$ is a bounded linear operator, thus there exists a unique bounded linear operator $\phi^*(\cdot):(L^2_{\mathcal F}(0,T;\mathbb R^n),\cdots,L^2_{\mathcal F}(0,T;\mathbb R^n))\rightarrow (L^2_{\mathcal F}(0,T;\Gamma),\cdots, L^2_{\mathcal F}(0,T;\Gamma))$ such that for any $u(\cdot)\in(L^2_{\mathcal F}(0,T;\Gamma),\cdots, L^2_{\mathcal F}(0,T;\Gamma))$ and \\$\mathbf x(\cdot)\in(L^2_{\mathcal F}(0,T;\mathbb R^n),\cdots,L^2_{\mathcal F}(0,T;\mathbb R^n))$,
$$\mathbb E\int_0^T\langle \phi(u)(t),\mathbf x(t)\rangle dt=\mathbb E\int_0^T\langle u(t),\phi^*(\mathbf x)(t)\rangle dt.$$
Hence, we can rewrite the cost functional as follows:
\begin{equation*}
  \begin{aligned}
2\mathcal{J}^{(N)}_{soc}(u)   %= &\mathbb{E}\int_{0}^{T}\Big[ \langle \phi^*\mathbf Q\phi (u),u\rangle + 2\langle \phi^*\mathbf Q\mathcal{S} (y),u\rangle +\langle\mathbf  Q\mathcal{S} (y),\mathcal{S} (y)\rangle + \langle\mathbf{R}u,u\rangle \Big] dt\\
  = &\mathbb{E}\int_{0}^{T} \Big[\langle (\phi^*\mathbf Q\phi + \mathbf{R})u,u\rangle + 2\langle \phi^*\mathbf Q\mathcal{S} (y) ,u\rangle +\langle\mathbf Q\mathcal{S} (y),\mathcal{S} (y)\rangle \Big]dt \\
  := & \langle M_2(u)(\cdot),u(\cdot)\rangle + 2\langle M_1,u(\cdot)\rangle  + M_0 ,\\
  \end{aligned}
\end{equation*}
where we have used $\langle\cdot,\cdot\rangle$ as inner products in different Hilbert spaces.
Note that, $M_2(\cdot)$ is a  bounded self-adjoint positive semi-definite linear operator.

\subsection{Agent $\mathcal A_i$ perturbation}\label{Minor agent's perturbation}

Let $\widetilde {\mathbf u}=(\widetilde u_1,\cdots,\widetilde u_N)$ be decentralized strategy given by
\begin{equation}\label{asymptotic optimal strategy}
\widetilde u_i(t)=\varphi_{\Theta_i}( p_i(t), q_i(t)):=\mathbf P_{\Gamma}[R(t)^{-1}\mathbb E[B(t)^\top p_i(t)+D_{\Theta_i}(t)^\top q_i(t)|\mathcal G_t^i]],\ i=1,\cdots,N,
\end{equation}
 where
 $(p_i,q_i)$ is the solution of
 \begin{equation*}\left\{\begin{aligned}
 &dx_i=[A_{\Theta_i}x_i+B\varphi_{\Theta_i}( p_i, q_i)+F\mathbb E\alpha]dt+[Cx_i+D_{\Theta_i}\varphi_{\Theta_i}( p_i, q_i)+\widetilde F\mathbb E\alpha]dW_i(t),\\
&dp_i=[-Qx_i+(QH+ H Q-HQH)\mathbb E\alpha-A_{\Theta_i}^\top p_i+F^\top \int_\mathcal{S}\check y_2^\theta d\Phi(\theta)+F^\top \mathbb E\check{y}_1\\
&\qquad\quad-C^\top q_i+\widetilde F^\top\mathbb E\check{\beta}_1]dt+q_idW_i(t),\\
 &x_i(0)=\xi,\quad p_i(T)=0,\quad i=1,\cdots,N.
\end{aligned}\right.\end{equation*}
 Here, $(\alpha,\check{y}_1,\check{\beta}_1,\check{y}_2^\theta)$ is the solution of \eqref{CC}.
Correspondingly, the realized decentralized states $(\widetilde x_1,\cdots,\widetilde x_N)$ satisfy
\begin{equation}\label{optimal state-realized}\left\{\begin{aligned}
&d\widetilde x_i=[A_{\Theta_i}\widetilde x_i+B\varphi_{\Theta_i}( p_i, q_i)+F\widetilde x^{(N)}]dt+[C
\widetilde x_i+D_{\Theta_i}\varphi_{\Theta_i}( p_i, q_i)+\widetilde F\widetilde x^{(N)}]dW_i(t),\\
&\widetilde x_i(0)=\xi,
\end{aligned}\right.\end{equation}
and $\widetilde x^{(N)}=\frac{1}{N}\sum_{i=1}^N\widetilde x_i$.

Let us consider the case that the agent $\mathcal A_i$ (without loss of generality, assume $i>1$) uses an alternative strategy $u_i\in\mathcal U_i^{c,f}$ while the other agents $\mathcal A_j,j\neq i$ use the strategy $\widetilde u_{-i}$. The realized state with the $i$-th agent's perturbation is
\begin{equation*}\left\{\begin{aligned}
&d  \acute{x}_i=[A_{\Theta_i} \acute x_i+B u_i+F \acute x^{(N)}]dt+[C  \acute x_i+D_{\Theta_i} u_i
+\widetilde  F \acute x^{(N)}]dW_i,\\
&d \acute x_j=[A_{\Theta_j} \acute x_j+B \varphi_{\Theta_j}( p_j, q_j)+F \acute x^{(N)}]dt+[C  \acute x_j+D_{\Theta_{j}}\varphi_j( p_j, q_j)
+\widetilde  F  \acute x^{(N)}]dW_j,\\
&\acute x_i(0)=\xi,\quad \acute x_j(0)=\xi,\quad 1\leq j\leq N,\quad j\neq i,
\end{aligned}\right.\end{equation*}
where $\acute x^{(N)}=\frac{1}{N}\sum_{i=1}^N\acute  x_i$. For $j=1,\cdots,N$, denote the perturbation
$$\delta u_i= u_i-\widetilde u_i,\quad \delta x_j=\acute x_j-\widetilde x_j,\quad \delta\mathcal J_j=\mathcal J_j(u_i,\widetilde u_{-i})-\mathcal J_j(\widetilde u_i,\widetilde u_{-i}).$$
Introducing the following frozen states
\begin{equation}\label{optimal state-frozen}\left\{\begin{aligned}
&d\widetilde l_j=[A_{\Theta_j}\widetilde l_j+B\varphi_{\Theta_j}( p_j, q_j)+F\mathbb E\alpha]dt+[C
\widetilde l_j+D_{\Theta_j}\varphi_{\Theta_j}( p_j, q_j)+\widetilde F\mathbb E\alpha]dW_j(t),\\
&\widetilde l_j(0)=\xi,\quad j=1,\cdots,N,
\end{aligned}\right.\end{equation}
and
\begin{equation}\nonumber\left\{\begin{aligned}
&d  \acute{l}_i=[A_{\Theta_i} \acute l_i+B u_i+F \mathbb E\alpha]dt+[C  \acute l_i+D_{\Theta_i} u_i
+\widetilde  F \mathbb E\alpha]dW_i,\\
&d \acute l_j=[A_{\Theta_j} \acute l_j+B \varphi_{\Theta_j}( p_j, q_j)+F \mathbb E\alpha]dt+[C  \acute l_j+D_{\Theta_{j}}\varphi_j( p_j, q_j)
+\widetilde  F  \mathbb E\alpha]dW_j,\\
&\acute l_i(0)=\xi,\quad \acute l_j(0)=\xi,\quad 1\leq j\leq N,\quad j\neq i.
\end{aligned}\right.\end{equation}
Similar to the computations in Section \ref{p-b-p optimality}, we have
\begin{equation*}\begin{aligned}
\delta \mathcal J_{soc}^{(N)}
 =\mathbb E\int_0^T&\Big[\langle Q\widetilde x_i,\delta x_i\rangle-\langle \Xi,\delta x_i\rangle
+\langle R\widetilde u_i,\delta u_i\rangle \Big] dt+\sum_{l=1}^{5}\epsilon_l,
\end{aligned}\end{equation*}
where
\begin{equation*}\left\{\begin{aligned}
&\epsilon_1=E\int_0^T\langle (QH+HQ- H Q H)(\mathbb E\alpha-\widetilde x^{(N)}),N\delta x^{(N)}\rangle dt,\\
&\epsilon_2=E\int_0^T\langle (QH+ H Q-HQH)\mathbb E\alpha,x^{**}-\delta x_{-i}\rangle dt,\\
&\epsilon_3=E\int_0^T\frac{1}{N}\sum_{j\neq i}\langle Q\widetilde x_j,N\delta x_j-x^*_j\rangle dt,\\
&\epsilon_{4}=\mathbb E\int_0^T\langle F^\top(\mathbb E[y_1^1]-\frac{1}{N}\sum_{j\neq i}y_1^j)+\widetilde F^\top(\mathbb E[\beta_{1}^{11}]-\frac{1}{N}\sum_{j\neq i}\beta_{1}^{jj}) ,\delta x_i\rangle dt,\\
&\epsilon_{5}=\mathbb E\int_0^T\langle F^\top(\mathbb E[y_1^1]-\frac{1}{N}\sum_{j\neq i}y_1^j)+\widetilde F^\top(\mathbb E[\beta_{1}^{11}]-\frac{1}{N}\sum_{j\neq i}\beta_{1}^{jj}) ,x^{**}\rangle dt.\\
\end{aligned}\right.\end{equation*}
Therefore, we have
\begin{equation*}\begin{aligned}
\delta \mathcal J_{soc}^{(N)}
 =\mathbb E\int_0^T&\Big[\langle Q\widetilde l_i,\delta l_i\rangle-\langle \Xi,\delta l_i\rangle
+\langle R\widetilde u_i,\delta u_i\rangle \Big] dt+\sum_{l=1}^{7}\epsilon_l,
\end{aligned}\end{equation*}
where
\begin{equation*}\left\{\begin{aligned}
&\epsilon_{6}=\mathbb E\int_0^T[\langle \acute l_i-\acute x_i,\Xi\rangle +\langle \widetilde l_i-\widetilde x_i,\Xi\rangle]dt,\\
&\epsilon_{7}=\mathbb E\int_0^T[\langle Q(\widetilde x_i-\widetilde l_i),\delta x_i\rangle+\langle Q\widetilde l_i,\acute x_i-\acute l_i\rangle+\langle Q\widetilde l_i, \widetilde x_i-\widetilde l_i\rangle ]dt.\\
\end{aligned}\right.\end{equation*}
First, we need some estimations. In the proofs, $L$ will denote a constant whose value may change from line to line. Applying the same technique as in \cite[Lemma 5.1]{HHN2018}, we have
\begin{lemma}\label{5.1}
There exist two constants $L_1$ and $L_2$ independent of $N$ such that
\begin{equation}\nonumber\begin{aligned}
&\mathbb E\sup_{0\leq t\leq T}\Big[|\alpha|^2+|\gamma|^2+|\check y_1|^2+|\check y_2^\theta|^2\Big]+\sum_{j=1}^N\mathbb E\sup_{0\leq t\leq T}\Big[|x_j|^2+|p_j|^2\Big]\\
&+\mathbb E\int_0^T\Big[|\vartheta|^2+|\check\beta_1|^2\Big]dt+\sum_{j=1}^N\mathbb E\int_0^T\Big[|q_j|^2+|\varphi_{\Theta_j}(p_j,q_j)|^2\Big]dt\leq L_1,
\end{aligned}\end{equation}
and
 \begin{equation}\nonumber
\sup_{1\leq j\leq N}\mathbb E\sup_{0\leq t\leq T}|\widetilde x_j(t)|^2+\sup_{1\leq j\leq N}\mathbb E\sup_{0\leq t\leq T}|\widetilde l_j(t)|^2\leq L_2.
\end{equation}
\end{lemma}

\begin{lemma}\label{le5.3}
There exists a constant $L_{3}$ independent of $N$ such that
\begin{equation}\nonumber
\mathbb E\sup_{0\leq s\leq t}|\delta x^{(N)}|^2+\sup_{1\leq j\leq N,j\neq i}\mathbb E\sup_{0\leq t\leq T}|\delta x_j|^2\leq \frac{L_{3}}{N^2}.
\end{equation}
\end{lemma}
\begin{proof}
Recall the equations \eqref{variation-i}, \eqref{minor-j} and \eqref{deltax-i}, we have
\begin{equation*}\begin{aligned}
\mathbb E\sup_{0\leq s\leq t}|\delta x_i|^2\leq L+L\mathbb E\int_0^t|\delta x_i|^2ds+L\mathbb E\int_0^t|\delta x^{(N)}|^2ds,
\end{aligned}\end{equation*}
\begin{equation}\label{variation-j}\begin{aligned}
\mathbb E\sup_{0\leq s\leq t}|\delta x_j|^2\leq L\mathbb E\int_0^t|\delta x_j|^2ds+L\mathbb E\int_0^t|\delta x^{(N)}|^2ds,
\end{aligned}\end{equation}
and
\begin{equation*}\begin{aligned}
\mathbb E\sup_{0\leq s\leq t}|\delta x_{-i}|^2\leq L\mathbb E\int_0^t|\delta x_{-i}|^2ds+LN^2\mathbb E\int_0^t|\delta x^{(N)}|^2ds.
\end{aligned}\end{equation*}
Note that $$\delta x^{(N)}=\frac{1}{N}\delta x_i+\frac{1}{N}\delta x_{-i},$$
we have
\begin{equation*}\begin{aligned}
\mathbb E\sup_{0\leq s\leq t}|\delta x_i|^2\leq L+L\mathbb E\int_0^t|\delta x_i|^2ds+\frac{L}{N^2}\mathbb E\int_0^t|\delta x_{-i}|^2ds,
\end{aligned}\end{equation*}
and
\begin{equation*}\begin{aligned}
\mathbb E\sup_{0\leq s\leq t}|\delta x_{-i}|^2\leq L\mathbb E\int_0^t|\delta x_{-i}|^2ds+L\mathbb E\int_0^t|\delta x_i|^2ds.
\end{aligned}\end{equation*}
Therefore, it follows from Gronwall inequality that
\begin{equation*}\begin{aligned}
\mathbb E\sup_{0\leq s\leq t}|\delta x_i|^2+\mathbb E\sup_{0\leq s\leq t}|\delta x_{-i}|^2\leq L.
\end{aligned}\end{equation*}
Thus,
\begin{equation*}
\mathbb E\sup_{0\leq s\leq t}|\delta x^{(N)}|^2\leq\frac{L}{N^2}.
\end{equation*}
From \eqref{variation-j}, by Gronwall inequality again,
we have
\begin{equation*}
\sup_{1\leq j\leq N,j\neq i}\mathbb E\sup_{0\leq s\leq t}|\delta x_j|^2\leq\frac{L}{N^2}.
\end{equation*}
\end{proof}

\begin{lemma}\label{le5.2}
There exists a constant $L_4$ independent of $N$ such that
\begin{equation}\nonumber
\sup_{0\leq t\leq T}\mathbb E|\widetilde x^{(N)}(t)-\mathbb E\alpha|^2\leq \frac{L_4}{N}.
\end{equation}
\end{lemma}

\begin{proof}
First, for any $\theta\in\mathcal S$, let
\begin{equation*}\left\{\begin{aligned}
&d\widetilde x_{\theta,j}=[A_{\theta}\widetilde x_{\theta,j}+B\varphi_{\theta}( p_j, q_j)+F\widetilde x_{\theta}^{(N)}]dt+[C
\widetilde x_{\theta,j}+D_{\theta}\varphi_{\theta}( p_j, q_j)+\widetilde F\widetilde x_{\theta}^{(N)}]dW_j(t),\\
&\widetilde x_{\theta,j}(0)=\xi,
\end{aligned}\right.\end{equation*}
\begin{equation*}\left\{\begin{aligned}
&d\widetilde l_{\theta,j}=[A_{\theta}\widetilde l_{\theta,j}+B\varphi_{\theta}( p_j, q_j)+F\mathbb E\alpha_{\theta}]dt+[C
\widetilde l_{\theta,j}+D_{\theta}\varphi_{\theta}( p_j, q_j)+\widetilde F\mathbb E\alpha_{\theta}]dW_j(t),\\
&\widetilde l_{\theta,j}(0)=\xi,
\end{aligned}\right.\end{equation*}
where $\widetilde x_\theta^{(N)}=\frac{1}{N}\sum_{j=1}^N\widetilde x_{\theta,j}$ and $\alpha_\theta$ is the solution of \eqref{CC} corresponding to $\Theta\equiv\theta$.
By Cauchy-Schwartz inequality and Burkholder-Davis-Gundy inequality, we have
\begin{equation*}\begin{aligned}
\mathbb E\sup_{0\leq s\leq t}|\widetilde x_{\theta,j}(s)-\widetilde l_{\theta,j}(s)|^2\leq
L\mathbb E\int_0^t[|\widetilde x_{\theta,j}(s)-\widetilde l_{\theta,j}(s)|^2+|\widetilde x_\theta^{(N)}(s)-\mathbb E\alpha_{\theta}(s)|^2]ds.
\end{aligned}\end{equation*}
By Gronwall inequality, we have
\begin{equation}\label{new-1}\begin{aligned}
\mathbb E\sup_{0\leq s\leq t}|\widetilde x_{\theta,j}(s)-\widetilde l_{\theta,j}(s)|^2\leq
L\mathbb E\int_0^t|\widetilde x_{\theta}^{(N)}(s)-\mathbb E\alpha_{\theta}(s)|^2ds.
\end{aligned}\end{equation}
Next, recalling the state equations \eqref{optimal state-realized} and \eqref{optimal state-frozen}, similarly we have
\begin{equation}\label{new-2}\begin{aligned}
\mathbb E\sup_{0\leq s\leq t}|\widetilde x_{j}(s)-\widetilde l_{j}(s)|^2\leq
L\mathbb E\int_0^t|\widetilde x^{(N)}(s)-\mathbb E\alpha(s)|^2ds.
\end{aligned}\end{equation}
Note that for any $t\in[0,T]$,
\begin{equation}\label{new-3}\begin{aligned}
&\mathbb E|\widetilde x^{(N)}(t)-\mathbb E\alpha(t)|^2\\
\leq&2\mathbb E|\frac{1}{N}\sum_{j=1}^N\widetilde x_j(t)-\frac{1}{N}\sum_{j=1}^N\int_{\mathcal S}\widetilde x_{\theta,j}(t)d\Phi(\theta)|^2\\
&+2\mathbb E|\frac{1}{N}\sum_{j=1}^N\int_{\mathcal S}\widetilde x_{\theta,j}(t)d\Phi(\theta)-\int_{\mathcal S}\mathbb E[\alpha(t)|\Theta=\theta]d\Phi(\theta)|^2\\
%\leq&6\mathbb E|\frac{1}{N}\sum_{j=1}^N\widetilde x_j(t)-\frac{1}{N}\sum_{j=1}^N\widetilde l_j(t)|^2+6
%\mathbb E|\frac{1}{N}\sum_{j=1}^N\widetilde l_j(t)-\frac{1}{N}\sum_{j=1}^N\int_{\mathcal S}\widetilde l_{\theta,j}(t)d\Phi(\theta)|^2\\
%&+6\mathbb E|\frac{1}{N}\sum_{j=1}^N\int_{\mathcal S}\widetilde l_{\theta,j}(t)d\Phi(\theta)-\frac{1}{N}\sum_{j=1}^N\int_{\mathcal S}\widetilde x_{\theta,j}(t)d\Phi(\theta)|^2\\
%&+2\mathbb E|\frac{1}{N}\sum_{j=1}^N\int_{\mathcal S}\widetilde x_{\theta,j}(t)d\Phi(\theta)-\int_{\mathcal S}\mathbb E[\alpha(t)|\Theta=\theta]d\Phi(\theta)|^2\\
\leq&\frac{6}{N}\sum_{j=1}^N\mathbb E|\widetilde x_j(t)-\widetilde l_j(t)|^2+\frac{6}{N^2}\sum_{j=1}^N\mathbb E|\widetilde l_j(t)-\int_{\mathcal S}\widetilde l_{\theta,j}(t)d\Phi(\theta)|^2\\
&+\frac{12}{N^2}\sum_{1\leq j\neq k\leq N}\langle\mathbb E(\widetilde l_j(t)-\int_{\mathcal S}\widetilde l_{\theta,j}(t)d\Phi(\theta)),\mathbb E(\widetilde l_k(t)-\int_{\mathcal S}\widetilde l_{\theta,k}(t)d\Phi(\theta))\rangle\\
&+6\mathbb E|\frac{1}{N}\sum_{j=1}^N\int_{\mathcal S}\widetilde l_{\theta,j}(t)d\Phi(\theta)-\frac{1}{N}\sum_{j=1}^N\int_{\mathcal S}\widetilde x_{\theta,j}(t)d\Phi(\theta)|^2\\
&+2\int_{\mathcal S}\mathbb E|\frac{1}{N}\sum_{j=1}^N\widetilde x_{\theta,j}(t)d\Phi(\theta)-\mathbb E[\alpha(t)|\Theta=\theta]|^2 d\Phi(\theta).\\
\end{aligned}\end{equation}
Similar to Lemma \ref{5.1}, there exists a constant $L$ such that
$$\sup_{\theta\in\mathcal S}\sup_{1\leq j\leq N}\mathbb E\sup_{0\leq t\leq T}|\widetilde x_{\theta,j}(t)|^2\leq L.$$
Consequently,
\begin{equation}\label{new-4}
\frac{6}{N^2}\sum_{j=1}^N\mathbb E|\widetilde l_j(t)-\int_{\mathcal S}\widetilde l_{\theta,j}(t)d\Phi(\theta)|^2\leq\frac{L}{N}.
\end{equation}
From $\mathbb E\alpha=\int_{\mathcal S}E\alpha_\theta d\Phi(\theta)$ and $\mathbb E(A_{\Theta_j}\widetilde l_j)=\int_{\mathcal S}\mathbb E(A_\theta\widetilde l_{\theta,j})d\Phi(\theta)$, we have
\begin{equation}\label{new-5}
\mathbb E(\widetilde l_j(t)-\int_{\mathcal S}\widetilde l_{\theta,j}(t)d\Phi(\theta))=0.
\end{equation}
It is easy to see that
\begin{equation}\label{new-6}\begin{aligned}
&\mathbb E|\frac{1}{N}\sum_{j=1}^N\int_{\mathcal S}\widetilde l_{\theta,j}(t)d\Phi(\theta)-\frac{1}{N}\sum_{j=1}^N\int_{\mathcal S}\widetilde x_{\theta,j}(t)d\Phi(\theta)|^2\\
=&\mathbb E|\frac{1}{N}\sum_{j=1}^N\int_{\mathcal S}(\widetilde l_{\theta,j}(t)-\widetilde x_{\theta,j}(t))d\Phi(\theta)|^2\\
%\leq&\frac{1}{N}\sum_{j=1}^N\mathbb E\int_{\mathcal S}|\widetilde l_{\theta,j}(t)-\widetilde x_{\theta,j}(t)|^2d\Phi(\theta)\\
\leq&\frac{1}{N}\sum_{j=1}^N\int_{\mathcal S}\mathbb E|\widetilde l_{\theta,j}(t)-\widetilde x_{\theta,j}(t)|^2d\Phi(\theta).
\end{aligned}\end{equation}
%and
%\begin{equation}
%\corO{\mathbb E\int_{\mathcal S}|\frac{1}{N}\sum_{j=1}^N\widetilde x_{\theta,j}(t)d\Phi(\theta)-\mathbb E[\alpha|\Theta=\theta]|^2 d\Phi(\theta)=\int_{\mathcal S}\mathbb E|\frac{1}{N}\sum_{j=1}^N\widetilde x_{\theta,j}(t)d\Phi(\theta)-\mathbb E[\alpha(t)|\Theta=\theta]|^2 d\Phi(\theta).}
%\end{equation}
Substituting \eqref{new-1}, \eqref{new-2}, \eqref{new-4}, \eqref{new-5}, and \eqref{new-6} into \eqref{new-3}, we have
\begin{equation*}\begin{aligned}
&\mathbb E|\widetilde x^{(N)}(t)-\mathbb E\alpha(t)|^2\\
\leq&L\mathbb E\int_0^t|\widetilde x^{(N)}(s)-\mathbb E\alpha(s)|^2ds+\frac{L}{N}+\frac{L}{N}\sum_{j=1}^N\int_{\mathcal S}\mathbb E\int_0^t|\widetilde x_{\theta}^{(N)}(s)-\mathbb E\alpha_{\theta}(s)|^2ds d\Phi(\theta)\\
&+2\int_{\mathcal S}\mathbb E|\frac{1}{N}\sum_{j=1}^N\widetilde x_{\theta,j}(t)-\mathbb E[\alpha(t)|\Theta=\theta]|^2 d\Phi(\theta).
\end{aligned}\end{equation*}
Applying similar method as homogeneous case (e.g. \cite[Lemma 6.3]{QHX}), we have
\begin{equation*}
\mathbb E|\frac{1}{N}\sum_{j=1}^N\widetilde x_{\theta,j}(t)-\mathbb E[\alpha(t)|\Theta=\theta]|^2\leq \frac{L}{N},
\end{equation*}
and
\begin{equation*}
\mathbb E\int_0^t|\widetilde x_{\theta}^{(N)}(s)-\mathbb E\alpha_{\theta}(s)|^2ds\leq\frac{L}{N}.
\end{equation*}
Therefore, there exists a constant $L$ independent of $t$ such that
\begin{equation*}\begin{aligned}
\mathbb E|\widetilde x^{(N)}(t)-\mathbb E\alpha|^2\leq&L\mathbb E\int_0^t|\widetilde x^{(N)}(s)-\mathbb E\alpha(s)|^2ds+\frac{L}{N}.
\end{aligned}\end{equation*}
By Gronwall inequality, we have
\begin{equation*}\begin{aligned}
\mathbb E|\widetilde x^{(N)}(t)-\mathbb E\alpha|^2\leq&\frac{L}{N}e^{Lt}.
\end{aligned}\end{equation*}
\end{proof}

\begin{lemma}\label{le5.4}
There exist constants $L_{5},L_{6}$ independent of $N$ such that
\begin{equation}\label{estimation-6}
\sup_{0\leq t\leq T}\mathbb E|x^{**}-\delta x_{-i}|^2\leq \frac{L_{5}}{N},
\end{equation}
and for $j\neq i$,
\begin{equation}\label{estimation-7}
\mathbb E\sup_{0\leq t\leq T}|N\delta x_j-x^*_j|^2\leq \frac{L_{6}}{N}.
\end{equation}
\end{lemma}
\begin{proof}
Introduce the following equations
\begin{equation*}\left\{\begin{aligned}
&d\delta\check x_i=[A_{\Theta_i}\delta\check x_i+B\delta u_i+\frac{F}{N}\delta x_{i}+\frac{F}{N}x^{**}]dt+[C\delta\check x_i+D_{\Theta_i}\delta u_i+\frac{\widetilde F}{N}\delta x_{i}+\frac{\widetilde F}{N}x^{**}]dW_i,\\
&j \neq i, \quad d\delta\check x_j=[A_{\Theta_j}\delta\check x_j+\frac{F}{N}\delta x_{i}+\frac{F}{N}x^{**}]dt+[C\delta\check x_j+\frac{\widetilde F}{N}\delta x_{i}+\frac{\widetilde F}{N}x^{**}]dW_j,\\
&\delta\check x_i(0)=0,\delta\check x_j(0)=0.
\end{aligned}\right.\end{equation*}
Recalling \eqref{minor-j}, by Cauchy-Schwartz inequality, Burkholder-Davis-Gundy inequality,
and Gronwall inequality, we have
\begin{equation}\label{new-8}\begin{aligned}
\mathbb E\sup_{0\leq s\leq t}|\delta x_{j}(s)-\delta\check x_{j}(s)|^2\leq
\frac{L}{N^2}\mathbb E\int_0^t|\delta x_{-i}(s)-x^{**}(s)|^2ds.
\end{aligned}\end{equation}
For any $\theta\in\mathcal S$, let
\begin{equation*}\left\{\begin{aligned}
&d\delta x_{\theta,i}=[A_{\theta}\delta x_{\theta,i}+B\delta u_i+F\delta x_\theta^{(N)}]dt+[C\delta x_{\theta,i}+D_{\theta}\delta u_i+\widetilde F\delta x_\theta^{(N)}]dW_i,\delta x_{\theta,i}(0)=0,\\
&j \neq i, \quad d\delta x_{\theta,j}=[A_{\theta}\delta x_{\theta,j}+F\delta x_\theta^{(N)}]dt+[C\delta x_{\theta,j}+\widetilde F\delta x_\theta^{(N)}]dW_j,
\delta x_{\theta,j}(0)=0,
\end{aligned}\right.\end{equation*}
and
\begin{equation*}\left\{\begin{aligned}
&d\delta\check x_{\theta,i}=[A_{\theta}\delta\check x_{\theta,i}+B\delta u_i+\frac{F}{N}\delta x_{\theta,i}+\frac{F}{N}x_\theta^{**}]dt+[C\delta x_{\theta,i}+D_{\theta}\delta u_i+\frac{\widetilde F}{N}\delta x_{\theta,i}+\frac{\widetilde F}{N}x_\theta^{**}]dW_i,\\
&  d\delta\check x_{\theta,j}=[A_{\theta}\delta\check x_{\theta,j}+\frac{F}{N}\delta x_{\theta,i}+\frac{F}{N}x_\theta^{**}]dt+[C\delta\check x_{\theta,j}+\frac{\widetilde F}{N}\delta x_{\theta,i}+\frac{\widetilde F}{N}x_\theta^{**}]dW_j,\\
&\delta\check x_{\theta,i}(0)=0,\ \delta\check x_{\theta,j}(0)=0,\ j \neq i,
\end{aligned}\right.\end{equation*}
where $\delta x_\theta^{(N)}=\frac{1}{N}\sum_{j=1}^N\delta x_{\theta,j}$.
Similarly,
\begin{equation}\label{new-9}\begin{aligned}
\mathbb E\sup_{0\leq s\leq t}|\delta x_{\theta,j}(s)-\delta\check x_{\theta,j}(s)|^2\leq
\frac{L}{N^2}\mathbb E\int_0^t|\sum_{j\neq i}\delta x_{\theta,j}(s)-x_\theta^{**}(s)|^2ds.
\end{aligned}\end{equation}
For any $t\in[0,T]$,
\begin{equation}\label{new-10}\begin{aligned}
&\mathbb E|x^{**}(t)-\delta x_{-i}(t)|^2\\
%\leq&
%2\mathbb E|\sum_{j\neq i}\delta x_{j}-\sum_{j\neq i}\int_\mathcal{S} \delta x_{\theta,j}d\Phi(\theta)|^2+2\mathbb E|\sum_{j\neq i}\int_\mathcal{S} \delta x_{\theta,j}d\Phi(\theta)-\int_\mathcal{S} x_\theta^{**}d\Phi(\theta)|^2\\
%\leq&6\mathbb E|\sum_{j\neq i}\delta x_{j}-\sum_{j\neq i}\delta\check x_{j}|^2+6\mathbb E|\sum_{j\neq i}\delta\check x_{j}-\sum_{j\neq i}\int_\mathcal{S} \delta \check x_{\theta,j}d\Phi(\theta)|^2+6\mathbb E|\sum_{j\neq i}\int_\mathcal{S} \delta \check x_{\theta,j}d\Phi(\theta)-\sum_{j\neq i}\int_\mathcal{S} \delta x_{\theta,j}d\Phi(\theta)|^2\\
%&+2\mathbb E|\int_\mathcal{S} \sum_{j\neq i}\delta x_{\theta,j}d\Phi(\theta)-\int_\mathcal{S} x_\theta^{**}d\Phi(\theta)|^2\\
\leq&6(N-1)\sum_{j\neq i}\mathbb E|\delta x_{j}-\delta\check x_{j}|^2+6\sum_{j\neq i}\mathbb E|\delta\check x_{j}-\int_\mathcal{S} \delta \check x_{\theta,j}d\Phi(\theta)|^2\\
&+12\sum_{1\leq j\neq k\leq N,j,k\neq i}\mathbb E\langle\delta\check x_{j}-\int_\mathcal{S} \delta \check x_{\theta,j}d\Phi(\theta),\mathbb \delta\check x_{k}-\int_\mathcal{S} \delta \check x_{\theta,k}d\Phi(\theta)\rangle\\
&+
6(N-1)\sum_{j\neq i}\int_\mathcal{S}\mathbb E| \delta \check x_{\theta,j}-\delta x_{\theta,j}|^2 d\Phi(\theta)+2\int_\mathcal{S}\mathbb E| \sum_{j\neq i}\delta x_{\theta,j}- x_\theta^{**}|^2 d\Phi(\theta).
\end{aligned}\end{equation}
Similar to Lemma \ref{le5.2}, we have
\begin{equation*}\begin{aligned}
&\mathbb E|x^{**}(t)-\delta x_{-i}(t)|^2\\
\leq&L\mathbb E\int_0^t|\delta x_{-i}(s)-x^{**}(s)|^2ds+\frac{L}{N}+
L\int_\mathcal{S}\mathbb E\int_0^t|\sum_{j\neq i}\delta x_{\theta,j}(s)-x_\theta^{**}(s)|^2ds d\Phi(\theta)\\
&+2\int_\mathcal{S}\mathbb E| \sum_{j\neq i}\delta x_{\theta,j}- x_\theta^{**}|^2 d\Phi(\theta).
\end{aligned}\end{equation*}
Applying similar technique as homogeneous case (e.g., pp. 29 in \cite{QHX}), we have
$$\mathbb E\sup_{0\leq s\leq t}|\sum_{j\neq i}\delta x_{\theta,j}(s)-x_\theta^{**}|^2(s)\leq\frac{L}{N}.$$
Therefore, there exists a constant $L$ independent of $t$ such that
\begin{equation*}\begin{aligned}
\mathbb E|x^{**}(t)-\delta x_{-i}(t)|^2
\leq L\mathbb E\int_0^t|\delta x_{-i}(s)-x^{**}(s)|^2ds+\frac{L}{N}.
\end{aligned}\end{equation*}
By Gronwall inequality, we have
\begin{equation*}\begin{aligned}
\mathbb E|x^{**}(t)-\delta x_{-i}(t)|^2
\leq \frac{L}{N}e^{Lt}.
\end{aligned}\end{equation*}
Hence \eqref{estimation-6} follows. Note that
\begin{equation*}\left\{\begin{aligned}
d(x_j^{*}-N\delta x_{j})=&[A_{\Theta_j}(x_j^{*}-N\delta x_j)+F(x^{**}-\delta x_{-i})]dt\\
&+[C(x_j^{*}-N\delta x_j)+\widetilde F(x^{**}-\delta x_{-i})]dW_j,\\
(x_j^{*}-\delta x_{j})(0)=&0.
\end{aligned}\right.\end{equation*}
By \eqref{estimation-6}, we have \eqref{estimation-7}.
\end{proof}

\begin{lemma}\label{le5.5}
There exists a constant $L_{7}$ independent of $N$ such that
\begin{equation}\label{estimation-8}
\sup_{1\leq j\leq N}\mathbb E\sup_{0\leq t\leq T}|\widetilde l_j-\widetilde x_{j}|^2\leq \frac{L_{7}}{N}.
\end{equation}
\end{lemma}
\begin{proof}
Note that
\begin{equation*}\left\{\begin{aligned}
&d(\widetilde l_j-\widetilde x_j)=[A_{\Theta_j}(\widetilde l_j-\widetilde x_j)+F(\mathbb E\alpha-\widetilde x^{(N)})]dt+[C(\widetilde l_j-
\widetilde x_j)+\widetilde F(\mathbb E\alpha-\widetilde x^{(N)})]dW_j(t),\\
&\widetilde l_j(0)-\widetilde x_j(0)=0.
\end{aligned}\right.\end{equation*}
By Cauchy-Schwartz inequality, Burkholder-Davis-Gundy inequality, Gronwall inequality and Lemma \ref{le5.2}, we have \eqref{estimation-8}.
\end{proof}

\subsection{Asymptotic optimality}

In order to prove asymptotic optimality, it suffices to consider the perturbations $u_{i}\in\mathcal U_i^c$ such that $\mathcal J_{soc}^{(N)}(u_1,\cdots,u_N)\leq\mathcal J_{soc}^{(N)}(\widetilde u_1,\cdots,\widetilde u_N).$ It is easy to check that $$\mathcal J_{soc}^{(N)}(\widetilde u_1,\cdots,\widetilde u_N)\leq LN,$$ where $L$ is a constant independent of $N$. Therefore, in the following we only consider the perturbations $u_i\in\mathcal U_i^{c}$ satisfying
\begin{equation}\nonumber
\sum_{i=1}^N\mathbb E\int_0^T|u_i|^2dt\leq LN.
\end{equation}
Therefore, similar to Lemma \ref{le5.2} and Lemma \ref{le5.5}, we have
\begin{lemma}\label{le5.6}
There exist constants $L_{8}$ and $L_9$ independent of $N$ such that
\begin{equation}\nonumber
\mathbb E\sup_{0\leq t\leq T}|\acute x^{(N)}(t)-\mathbb E\alpha|^2\leq \frac{L_8}{N},
\end{equation}
and
\begin{equation}\nonumber
\sup_{1\leq j\leq N}\mathbb E\sup_{0\leq t\leq T}|\acute l_j-\acute x_{j}|^2\leq \frac{L_{9}}{N}.
\end{equation}
\end{lemma}
Let $\delta u_i=u_i-\widetilde u_i$,
 and consider a perturbation  ${u} = \widetilde{u} + (\delta u_1,\cdots,\delta u_N):=\widetilde{ u}+ \delta u$. Then by Section \ref{Representation of social cost}, we have
\begin{equation*}
  \begin{aligned}
2\mathcal{J}^{(N)}_{soc}(\widetilde{u} + \delta u) =&  \langle M_2(\widetilde{u} + \delta u),\widetilde{u} + \delta u\rangle + 2\langle M_1,\widetilde{u} + \delta u\rangle  + M_0\\
= &2\mathcal{J}^{(N)}_{soc}(\widetilde{u} )  + 2\sum_{i=1}^{N}\langle M_2(\widetilde{u}) + M_1 , \delta u_i\rangle + \langle M_2 (\delta u), \delta u\rangle, \\
  \end{aligned}
\end{equation*}
where $M_2(\widetilde{u} )+ M_1 $ is the Fr\'{e}chet differential of ${\mathcal J}^{(N)}_{soc}$ on $\widetilde{u}$.
\begin{theorem}\label{asymptotic optimal}
Under the assumptions \emph{(A1)-(A5)}, $\widetilde u=(\widetilde u_1,\cdots,\widetilde u_N)$ defined in \eqref{asymptotic optimal strategy} is a $\Big(\frac{1}{\sqrt{N}}\Big)$-social optimal strategy for the agents.
\end{theorem}
\begin{proof}
From Section \ref{Minor agent's perturbation}, we have
\begin{equation*}
\langle M_2(\widetilde u)+M_1,\delta u_i\rangle =\mathbb E\int_0^T\Big[\langle Q\widetilde l_i,\delta l_i\rangle-\langle \Xi,\delta l_i\rangle+\langle R\widetilde u_i,\delta u_i\rangle \Big] dt+\sum_{l=1}^{7}\varepsilon_l.
\end{equation*}
From the optimality of $\widetilde u$, we have
$$\mathbb E\int_0^T\Big[\langle Q\widetilde l_i,\delta l_i\rangle-\langle \Xi,\delta l_i\rangle+\langle R\widetilde u_i,\delta u_i\rangle \Big] dt\geq0.$$
Suppose this is not true, then for $u_i$ such that $\widetilde u_i+u_i\in\mathcal U_i^{d,p}$, we have $$\widetilde u_i+\rho u_i\in\mathcal U_i^{d,p},\qquad 0<\rho<1,$$and $$\lim_{\rho\rightarrow0}\frac{ J_i(\widetilde u_i+\rho u_i,\widetilde u_{-i})- J_i(\widetilde u_i,\widetilde u_{-i})}{\rho}<0.$$
Therefore, $$ J_i(\widetilde u_i+\rho u_i,\widetilde u_{-i})< J_i(\widetilde u_i,\widetilde u_{-i})$$ for sufficiently small $\rho$, which is a contradiction with the optimality of $\widetilde u_i$.
Moreover, combing Lemmas \ref{le5.2}-\ref{le5.6} with iteration analysis (e.g., \cite{HWY2019}), we have %$$\varepsilon_1+\varepsilon_2+\varepsilon_3+\varepsilon_6+\varepsilon_7=
%O\Big(\frac{1}{\sqrt{N}}\Big).$$
$$\sum_{l=1}^{7}\varepsilon_l=O\Big(\frac{1}{\sqrt{N}}\Big).$$
Therefore,
\begin{equation*}
  \begin{aligned}
&\mathcal{J}^{(N)}_{soc}(\widetilde{u} + \delta u)\\
% = &\mathcal{J}^{(N)}_{soc}(\widetilde{u} )  + \sum_{i=1}^{N}\langle M_2(\widetilde{u}) + M_1 , \delta u_i\rangle +\frac{1}{2} \langle M_2 (\delta u), \delta u\rangle \\
=&\mathcal{J}^{(N)}_{soc}(\widetilde{u} )  + \sum_{i=1}^{N}\mathbb E\int_0^T\Big[\langle Q\widetilde l_i,\delta l_i\rangle-\langle \Xi,\delta l_i\rangle+\langle R\widetilde u_i,\delta u_i\rangle \Big] dt+ \sum_{i=1}^{N}\sum_{l=1}^{5}\varepsilon_l+\frac{1}{2} \langle M_2 (\delta u), \delta u\rangle. \\
  \end{aligned}
\end{equation*}
Note that $$\sum_{i=1}^{N}\mathbb E\int_0^T\Big[\langle Q\widetilde l_i,\delta l_i\rangle-\langle \Xi,\delta l_i\rangle+\langle R\widetilde u_i,\delta u_i\rangle \Big] dt+\frac{1}{2} \langle M_2 (\delta u), \delta u\rangle \geq 0,$$and
$$\sum_{i=1}^{N}\sum_{l=1}^{7}\varepsilon_l=O(\sqrt{N}),$$
there exists a constant $L$ independent of $N$ such that
$$\frac{1}{N}\Big(\mathcal J_{soc}^{(N)}
(\widetilde u)-\inf_{u\in\mathcal U_i^c}\mathcal J_{soc}^{(N)}(u)\Big)\leq \frac{L}{\sqrt{N}}.$$
\end{proof}

%%%%%%%%%%%%%%%%%%%%%%%%%%%%%%%%%%%%%%%%%%%%%%
%% Example with single Appendix:            %%
%%%%%%%%%%%%%%%%%%%%%%%%%%%%%%%%%%%%%%%%%%%%%%
\begin{appendix}
\section*{}\label{appn} %% if no title is needed, leave empty \section*{}.
First, for any given $(Y,Z)\in L_{\mathbb F}^2(0,T;\mathbb{R}^m)\times L_{\mathbb F}^2(0,T;\mathbb R^{m})$ and $0\leq t\leq T$, the following SDE has a unique solution:
\begin{equation}\label{1}
X(t)=x+\int_0^tb(s,X,\mathbb E[X],Y,\mathcal E_t[Y],Z,\mathcal E_t[Z])ds+\int_0^t\sigma(s,X,\mathbb E[X],Y,\mathcal E_t[Y],Z,\mathcal E_t[Z])d W(s).
\end{equation}
Therefore, we can introduce a map $\mathcal M_1:L^2_{\mathbb F}(0,T;\mathbb R^m)\times L^2_{\mathbb F}(0,T;\mathbb R^{m})\rightarrow L^2_{\mathbb F}(0,T;\mathbb R^n)$.
Moreover, by the standard estimations of SDE, we have the following result:
\begin{lemma}
Let $X_i$ be the solution of \eqref{1} corresponding to $(Y_i,Z_i)$, $i=1,2$ respectively. Then for all $\rho\in\mathbb R$ and some constants $l_1,l_2,l_3,l_4>0$, we have
\begin{equation}\nonumber\begin{aligned}
&\mathbb Ee^{-\rho t}|\hat X(t)|^2+\bar\rho_1\mathbb E\int_0^te^{-\rho s}|\hat X(s)|^2ds\\
\leq&(k_2l_1+k_3l_2+k_{14}^2+k_{15}^2)\mathbb E\int_0^te^{-\rho s}|\hat Y(s)|^2ds\\&+(k_4l_3+k_5l_4+k_{12}^2+k_{16}^2+k_{17}^2)\mathbb E\int_0^te^{-\rho s}|\hat Z(s)|^2ds,
\end{aligned}\end{equation} and
\begin{equation}\nonumber\begin{aligned}
\mathbb Ee^{-\rho t}|\hat X(t)|^2
\leq&(k_2l_1+k_3l_2+k_{14}^2+k_{15}^2)\mathbb E\int_0^te^{-\bar\rho_1(t-s)-\rho s}|\hat Y(s)|^2ds\\
&+(k_4l_3+k_5l_4+k_{12}^2+k_{16}^2+k_{17}^2)\mathbb E\int_0^te^{-\bar\rho_1(t-s)-\rho s}|\hat Z(s)|^2ds,
\end{aligned}\end{equation}
where $\bar\rho_1=\rho-2\rho_1-2k_1-k_2l_1^{-1}-k_3l_2^{-1}-k_4l_3^{-1}-k_5l_4^{-1}-k_{12}^2-k_{13}^2$ and $\hat\Phi:=\Phi_1-\Phi_2$, $\Phi=X,Y,Z$. Moreover,
\begin{equation}\nonumber\begin{aligned}
\mathbb E\int_0^Te^{-\rho t}|\hat X(t)|^2dt
\leq&(k_2l_1+k_3l_2+k_{14}^2+k_{15}^2)\frac{1-e^{-\bar\rho_1T}}{\bar\rho_1}\mathbb E\int_0^Te^{-\rho s}|\hat Y(s)|^2ds\\&
+(k_4l_3+k_5l_4+k_{12}^2+k_{16}^2+k_{17}^2)\frac{1-e^{-\bar\rho_1T}}{\bar\rho_1}\mathbb E\int_0^Te^{-\rho s}|\hat Z(s)|^2ds,
\end{aligned}\end{equation}
and
\begin{equation}\nonumber\begin{aligned}
e^{-\rho T}\mathbb E|\hat X(T)|^2
\leq&(1\vee e^{-\bar\rho_1T})\Big\{(k_2l_1+k_3l_2+k_{14}^2+k_{15}^2)\mathbb E\int_0^Te^{-\rho t}|\hat Y(t)|^2dt\\
&+(k_4l_3+k_5l_4+k_{12}^2+k_{16}^2+k_{17}^2)\mathbb E\int_0^Te^{-\rho t}|\hat Z(t)|^2dt\Big\}.
\end{aligned}\end{equation}
Specially, if $\bar\rho_1>0$,
\begin{equation}\begin{aligned}\nonumber
e^{-\rho T}\mathbb E|\hat X(T)|^2\leq&(k_2l_1+k_3l_2+k_{14}^2+k_{15}^2)\mathbb E\int_0^Te^{-\rho t}|\hat Y(t)|^2dt\\
&+(k_4l_3+k_5l_4+k_{12}^2+k_{16}^2+k_{17}^2)\mathbb E\int_0^Te^{-\rho t}|\hat Z(t)|^2dt.
\end{aligned}\end{equation}
\end{lemma}
Next, for any given $X\in L_{\mathbb F}^2(0,T;\mathbb R^n)$, consider the following BSDE:
\begin{equation}\label{7}
Y(t)=\int_t^Tf(s,X,\mathbb E[X],Y,\mathbb E[Y],\widetilde{\mathbb E}[Y],Z,\mathbb E[Z])ds-\int_t^TZ(s)dW(s).
\end{equation}
\begin{proposition}
\eqref{7} admits a unique solution $(Y,Z)\in L^2_{\mathbb F}(0,T;\mathbb R^m)\times L^2_{\mathbb F}(0,T;\mathbb R^{m})$.
\end{proposition}
\begin{proof}
For any fixed $(y,z)\in L^2_{\mathbb F}(0,T;\mathbb R^m)\times L^2_{\mathbb F}(0,T;\mathbb R^{m})$,
\begin{equation}\nonumber
Y(t)=\int_t^Tf(s,X,\mathbb E[X],Y,\mathbb E[y],\widetilde{\mathbb E}[y],z,\mathbb E[z])ds-\int_t^TZ(s)dW(s)
\end{equation}
admits a unique solution $(Y,Z)\in L^2_{\mathbb F}(0,T;\mathbb R^m)\times L^2_{\mathbb F}(0,T;\mathbb R^{m})$. Hence we can introduce the mapping $\mathcal N:(y,z)\rightarrow(Y,Z)$. For any $(y,z),(y',z')\in L^2_{\mathbb F}(0,T;\mathbb R^m)\times L^2_{\mathbb F}(0,T;\mathbb R^{m})$, denote $(Y,Z)=\mathcal N(y,z)$ and $(Y',Z')=\mathcal N(y',z')$. Let $(\hat y,\hat z,\hat Y,\hat Z)=(y-y',z-z',Y-Y',Z-Z')$. Applying It\^{o}'s formula to $e^{\delta x}|\hat Y(s)|^2$, we have
\begin{equation*}\begin{aligned}
&e^{\delta t}|\hat Y(t)|^2+\int_t^Te^{\delta s}|\hat Z(s)|ds+\int_t^T\delta e^{\delta s}|\hat Y(s)|ds\\
%=&2\int_t^T e^{\delta s}\langle\hat Y(s),f(s,X,\mathbb E[X],Y,\mathbb E[y],\widetilde{\mathbb E}[y],z,\mathbb E[z])-f
%(s,X,\mathbb E[X],Y',\mathbb E[y],\widetilde{\mathbb E}[y],z,\mathbb E[z])\rangle ds\\
%&+2\int_t^T e^{\delta s}\langle\hat Y(s),\hat Z(s)dW(s)\rangle\\
%\leq&2\int_t^Te^{\delta s}[\rho_2|\hat Y(s)|^2+|\hat Y(s)|(k_8\mathbb E[\hat y]+k_9\widetilde{\mathbb E}[\hat y]+k_{10}\hat z+k_{11}\mathbb E[\hat z])]ds+2\int_t^T e^{\delta s}\langle\hat Y(s),\hat Z(s)dW(s)\rangle\\
\leq&\int_t^Te^{\delta s}(2\rho_2+4k_8^2+4k_9^2+4k_{10}^2+4k_{11}^2)|\hat Y(s)|^2ds\\
&+\frac{1}{4}\int_t^Te^{\delta s}(\mathbb E[|\hat y|^2]+\widetilde{\mathbb E}[|\hat y^2|]+|\hat z|^2+\mathbb E[|\hat z|^2])ds+2\int_t^T e^{\delta s}\langle\hat Y(s),\hat Z(s)dW(s)\rangle.\\
\end{aligned}\end{equation*}
Note that $\mathbb E[\widetilde{\mathbb E}[|\hat y^2|]]=\mathbb E[|\hat y^2|]$, letting $\delta=2\rho_2+4k_8^2+4k_9^2+4k_{10}^2+4k_{11}^2$ and taking expectation, we have
$$\mathbb E\int_t^Te^{\delta s}(|\hat Y(s)|^2+|\hat Z(s)|^2)ds\leq \frac{1}{2}\mathbb E\int_t^Te^{\delta s}(|\hat y(s)|^2+|\hat z(s)|^2)ds,$$
i.e., $\mathcal N$ is a contraction mapping. Hence \eqref{7} admits a unique solution $(Y,Z)\in L^2_{\mathbb F}(0,T;\mathbb R^m)\times L^2_{\mathbb F}(0,T;\mathbb R^{m})$.
\end{proof}
Thus, we can introduce another map $\mathcal M_2:L^2_{\mathbb F}(0,T;\mathbb R^n)\rightarrow L^2_{\mathbb F}(0,T;\mathbb R^m)\times L^2_{\mathbb F}(0,T;\mathbb R^{m}).$
By the standard estimation of BSDE, we have the following result:
\begin{lemma}
Let $(Y_i,Z_i)$ be the solution of \eqref{7} corresponding to $X_i,i=1,2$, respectively. Then for all $\rho\in\mathbb R$ and some constants $l_5,l_6,l_7,l_8>0$, we have
\begin{equation*}\begin{aligned}
&\mathbb E e^{-\rho t}|\hat Y(t)|^2+\bar\rho_2\mathbb E\int_t^Te^{-\rho s}|\hat Y(s)|^2ds+(1-k_{10}l_7-k_{11}l_8)\mathbb E\int_t^Te^{-\rho s}|\hat Z(s)|^2ds\\
\leq&(k_6l_5+k_7l_6)\mathbb E\int_t^Te^{-\rho s}|\hat X(s)|^2ds,
\end{aligned}\end{equation*}
and
\begin{equation*}\begin{aligned}
&\mathbb E e^{-\rho t}|\hat Y(t)|^2+(1-k_7l_5-k_8l_6)\mathbb E\int_t^Te^{-\rho s}|\hat Z(s)|^2ds\\
\leq&(k_4l_3+k_5l_4)\mathbb E\int_t^Te^{-\bar\rho_2(s-t)-\rho s}|\hat X(s)|^2ds,
\end{aligned}\end{equation*}
where $\bar\rho_2=-\rho-2\rho_2-2k_8-2k_9-k_6l_5^{-1}-k_7l_6^{-1}-k_{10}l_7^{-1}-k_{11}l_8^{-1}$,
and $\hat\Phi:=\Phi_1-\Phi_2$, $\Phi=X,Y,Z$. Moreover,
\begin{equation*}\begin{aligned}
\mathbb E \int_0^Te^{-\rho t}|\hat Y(t)|^2dt
\leq\frac{1-e^{-\bar\rho_2T}}{\bar\rho_2}(k_6l_5+k_7l_6)\mathbb E\int_0^Te^{-\rho s}|\hat X(s)|^2ds,
\end{aligned}\end{equation*}
and
\begin{equation*}\begin{aligned}
\mathbb E \int_0^Te^{-\rho t}|\hat Z(t)|^2dt
\leq\frac{(k_6l_5+k_7l_6)(1\vee e^{-\bar\rho_2T})}{(1-k_{10}l_7-k_{11}l_8)(1\wedge e^{-\bar\rho_2T})}\mathbb E\int_0^Te^{-\rho s}|\hat X(s)|^2ds.
\end{aligned}\end{equation*}
Specially, if $\bar\rho_2>0$,
\begin{equation*}\begin{aligned}
\mathbb E \int_0^Te^{-\rho t}|\hat Z(t)|^2dt
\leq\frac{k_6l_5+k_7l_6}{1-k_{10}l_7-k_{11}l_8}\mathbb E\int_0^Te^{-\rho s}|\hat X(s)|^2ds.
\end{aligned}\end{equation*}

\end{lemma}

\medskip

\textbf{Proof of Theorem \ref{discounting}:}
Define $\mathcal M:=\mathcal M_2\circ\mathcal M_1$, where $\mathcal M_1$ is defined by \eqref{1} and $\mathcal M_2$ is defined by \eqref{7}. Thus $\mathcal M$ is a mapping from $L_{\mathbb F}^2(0,T;\mathbb R^m)\times L_{\mathbb F}^2(0,T;\mathbb R^{m})$ into itself.
For $(U_i,V_i)\in L_{\mathbb F}^2(0,T;\mathbb R^m)\times L_{\mathbb F}^2(0,T;\mathbb R^{m})$, let $X_i:=\mathcal M_1(U_i,V_i)$ and $(Y_i,Z_i):=\mathcal M(U_i,V_i)$.
Therefore,
\begin{equation*}\begin{aligned}
&\mathbb E\int_0^T e^{-\rho t}|Y_1(t)-Y_2(t)|^2dt+
\mathbb E\int_0^T e^{-\rho t}|Z_1(t)-Z_2(t)|^2dt\\
%\leq&\Big[\frac{1-e^{-\bar\rho_2T}}{\bar\rho_2}+
%\frac{1\vee e^{-\bar\rho_2T}}{(1-k_{10}l_7-k_{11}l_8)(1\wedge e^{-\bar\rho_2T})}\Big](k_6l_5+k_7l_6)\mathbb E\int_0^T e^{-\rho t}|X_1(t)-X_2(t)|^2dt\\
\leq&\Big[\frac{1-e^{-\bar\rho_2T}}{\bar\rho_2}+
\frac{1\vee e^{-\bar\rho_2T}}{(1-k_{10}l_7-k_{11}l_8)(1\wedge e^{-\bar\rho_2T})}\Big](k_6l_5+k_7l_6)\frac{1-e^{-\bar\rho_1T}}{\bar\rho_1}\\
&\Big\{(k_2l_1+k_3l_2+k_{14}^2+k_{15}^2)
\mathbb E\int_0^T e^{-\rho t}|U_1(t)-U_2(t)|^2dt\\&+
(k_4l_3+k_5l_4+k_{12}^2+k_{16}^2+k_{17}^2)\mathbb E\int_0^T e^{-\rho t}|V_1(t)-V_2(t)|^2dt\Big\}.
\end{aligned}\end{equation*}
Choosing suitable $\rho$, we get that $\mathcal M$ is a contraction mapping.

Furthermore, if $2\rho_1+2\rho_2<-2k_1-2k_8-2k_9-k_{10}^2-k_{11}^2-k_{12}^2-k_{13}^2$, we can choose $\rho\in\mathbb R$, $0<k_{10}l_7<\frac{1}{2}$ and $0<k_{11}l_8<\frac{1}{2}$ and sufficient large $l_1,l_2,l_3,l_4,l_5,l_6$ such that
$$\bar\rho_1>0,\qquad \bar\rho_2>0,\qquad 1-k_{10}l_7-k_{11}l_8>0.$$
Therefore,
\begin{equation*}\begin{aligned}
&\mathbb E\int_0^T e^{-\rho t}|Y_1(t)-Y_2(t)|^2dt+
\mathbb E\int_0^T e^{-\rho t}|Z_1(t)-Z_2(t)|^2dt\\
\leq&\Big[\frac{1}{\bar\rho_2}+
\frac{1}{1-k_{10}l_7-k_{11}l_8}\Big]\frac{1}{\bar\rho_1}
(k_6l_5+k_7l_6)\\&\Big\{(k_2l_1+k_3l_2+k_{14}^2+k_{15}^2)\mathbb E\int_0^T e^{-\rho t}|U_1(t)-U_2(t)|^2dt\\
&+(k_4l_3+k_5l_4+k_{12}^2+k_{16}^2+k_{17}^2)\mathbb E\int_0^T e^{-\rho t}|V_1(t)-V_2(t)|^2dt\Big\}.
\end{aligned}\end{equation*}
Thus, the proof is complete.

\qed
\end{appendix}

\end{document}